\documentclass[11pt,reqno]{amsart}
\usepackage{amsmath,amssymb,latexsym,soul,cite,mathrsfs}
\usepackage{color,enumitem,graphicx}
\usepackage[colorlinks=true,urlcolor=blue,
citecolor=red,linkcolor=blue,linktocpage,pdfpagelabels,
bookmarksnumbered,bookmarksopen]{hyperref}
\usepackage[english]{babel}
\usepackage{enumitem}

\usepackage[left=2.5cm,right=2.5cm,top=2.9cm,bottom=2.9cm]{geometry}

\usepackage[hyperpageref]{backref}

\makeatletter
\newcommand{\leqnomode}{\tagsleft@true}
\newcommand{\reqnomode}{\tagsleft@false}
\makeatother

\numberwithin{equation}{section}

\newtheorem{theorem}{Theorem}[section]
\newtheorem{lemma}[theorem]{Lemma}

\newtheorem{proposition}[theorem]{Proposition}

\newtheorem{remark}[theorem]{Remark}

\newtheorem{theoremletter}{Theorem}

\renewcommand{\rightarrow}{\to}

\newcommand{\ud}{\mathrm{d}}

\title[Existence of bound and ground states for fractional coupled systems in $\mathbb{R}^{N}$]{Existence of bound and ground states for fractional coupled systems in $\mathbb{R}^{N}$}

\author[J.C. \ de Albuquerque]{Jos\'e Carlos de Albuquerque}
\author[J.M.\ do \'O]{Jo\~ao Marcos do \'O}
\author[E.D. Silva]{Edcarlos Domingos Silva}

\address[J.C. de~Albuquerque]{Instituto de Matem\'{a}tica e Estat\'{i}stica,
	Universidade Federal de Goi\'{a}s
	\newline\indent
	74001-970, Goi\'{a}s-GO, Brasil}
\email{\href{mailto:joserre@gmail.com}{joserre@gmail.com}}

\address[J.M. do \'O]{Departamento de Matem\'{a}tica,
	Universidade Federal da Para\'{\i}ba
	\newline\indent
	58051-900, Jo\~ao Pessoa-PB, Brasil}
\email{\href{mailto:jmbo@pq.cnpq.br}{jmbo@pq.cnpq.br}}

\address[E.D. Silva]{Instituto de Matem\'{a}tica e Estat\'{i}stica,
	Universidade Federal de Goi\'{a}s}
\email{\href{mailto:eddomingos@hotmail.com}{eddomingos@hotmail.com}}

\thanks{Corresponding author: Edcarlos D. Silva}

\thanks{Research supported in part by INCTmat/MCT/Brazil, CNPq and CAPES/Brazil. The third author was also partially supported by Fapeg/CNpq grants 03/2015-PPP}

\date{\today}

\begin{document}
	

\begin{abstract}
	In this work we consider the following class of nonlocal linearly coupled systems involving Schr\"{o}dinger equations with fractional laplacian
	 $$
	  \left\{
	  \begin{array}{lr}
	  (-\Delta)^{s_{1}} u+V_{1}(x)u=f_{1}(u)+\lambda(x)v, & x\in\mathbb{R}^{N},\\
	  (-\Delta)^{s_{2}} v+V_{2}(x)v=f_{2}(v)+\lambda(x)u, & x\in\mathbb{R}^{N},
	  \end{array}
	  \right.
	 $$
where $(-\Delta)^{s}$ denotes de fractional Laplacian, $s_{1},s_{2}\in(0,1)$ and $N\geq2$. The coupling function $\lambda:\mathbb{R}^{N} \rightarrow \mathbb{R}$ is related with the potentials by $|\lambda(x)|\leq \delta\sqrt{V_{1}(x)V_{2}(x)}$, for some $\delta\in(0,1)$. We deal with periodic and asymptotically periodic bounded potentials. On the nonlinear terms $f_{1}$ and $f_{2}$, we assume ``superlinear" at infinity and at the origin. We use a variational approach to obtain the existence of bound and ground states without assuming the well known Ambrosetti-Rabinowitz condition at infinity. Moreover, we give a description of the ground states when the coupling function goes to zero.
\end{abstract}

\maketitle

	

\section{Introduction}	
In this work we consider the following class of linearly coupled systems involving Schr\"{o}dinger equations with fractional laplacian
\begin{equation}\label{j000}
\left\{
\begin{array}{lr}
(-\Delta)^{s_{1}} u+V_{1}(x)u=f_{1}(u)+\lambda(x)v, & x\in\mathbb{R}^{N},\\
(-\Delta)^{s_{2}} v+V_{2}(x)v=f_{2}(v)+\lambda(x)u, & x\in\mathbb{R}^{N},
\end{array}
\right.
\end{equation}
where $N\geq2$, $s_{1},s_{2}\in(0,1)$ and $V_{1},V_{2}$ are bounded and continuous potentials. Here $\lambda: \mathbb{R}^{N} \rightarrow \mathbb{R}$ is also a bounded and continuous function satisfying some suitable hypotheses. It is worthwhile to mention that a solution $(u,v)$ of finite energy for \eqref{j000} is called a \textit{bound state solution}. It is well known that $(u,v)\neq(0,0)$ is called \textit{ground state solution} if admits the smallest energy among all nontrivial bound states of \eqref{j000}. Our main contribution here is to consider fractional Schr\"{o}dinger equations which are linearly coupled finding existence of bound and ground state solutions. We deal with two classes of potentials: periodic and asymptotically periodic. Furthermore, we study the behavior of the ground state solutions when the coupling function goes to zero. By using variational arguments we get our main results taking into account that the nonlinear terms does not verify the well known Ambrosetti-Rabinowitz condition.


\subsection{Motivation and related results} In order to motivate our results we begin by giving
a brief survey on this subject. In the last few years, a great attention
has been focused on the study of problems involving fractional Sobolev spaces and corresponding nonlocal equations, both from a pure mathematical point of view and their concrete applications, since they naturally arise in many different contexts, such as, among the others, obstacle problems, flame propagation, minimal surfaces, conservation laws, financial market, optimization, crystal dislocation, phase transition and water waves, see for instance \cite{guia,caffa} and references therein.

Coupled elliptic systems arise in various branches of mathematical physics and nonlinear optics (see \cite{MR1212560}). Solutions of System~\eqref{j000} are related with standing wave solutions of the following two-component system of nonlinear equations
 \begin{equation}\label{j00}
 \left\{
 \begin{array}{ll}
 i\displaystyle\frac{\partial\psi}{\partial t}=(-\Delta)^{s_{1}}\psi+V_{1}(x)\psi-f_{1}(\psi)-\lambda(x)\phi, & x\in \mathbb{R}^{N}, \ t\geq0,\\
 i\displaystyle\frac{\partial\phi}{\partial t}=(-\Delta)^{s_{2}}\phi+V_{2}(x)\phi-f_{2}(\phi)-\lambda(x)\psi, & x\in \mathbb{R}^{N}, \ t\geq0,
 \end{array}
 \right.
 \end{equation}
where $i$ denotes the imaginary unit. For System~\eqref{j00}, a solution of the form $(\psi(x,t),\phi(x,t))=(e^{-it}u(x),e^{-it}v(x))$
is called \textit{standing wave}. Assuming that $f_{1}(e^{i\theta}u)=e^{i\theta}f_{1}(u)$ and $f_{2}(e^{i\theta}v)=e^{i\theta}f_{2}(v)$, for $u,v\in\mathbb{R}$, it is well known that $(\psi,\phi)$ is a solution of \eqref{j00} if and only if $(u,v)$ solves System~\eqref{j000}. For more information on the physical background we refer the readers to \cite{MR1212560,MR695535,bl2,l1,l2,MR0454365} and references therein.	

Notice that if $\lambda\equiv0$, $s_{1}=s_{2}=s$ and $V_{1}=V_{2}=V$, then System~\eqref{j000} reduces to the general class of nonlinear fractional Schr\"{o}dinger equations $(-\Delta)^{s} u+V(x)u=f(u)$, in $\mathbb{R}^{N}$. It is known that when $s\rightarrow1$, the fractional Laplacian $(-\Delta)^{s}$ reduces to the standard Laplacian $-\Delta$, see \cite{guia}. There is a huge bibliography concerned to nonlinear Schr\"{o}dinger equations, we refer the classical works \cite{rabino,giovany,wang,MR1306679} and references therein. Recently, fractional Schr\"{o}dinger equations have been studied under many different assumptions on the potential $V(x)$ and on the nonlinearity $f(u)$. For instance, in \cite{felmer}, in order to overcome the lack of compactness, the authors used a comparison argument to obtain positive solutions for the case when $V\equiv1$. Another way to overcome this difficulty is requiring coercive potentials, that is, $V(x)\rightarrow+\infty$, as $|x|\rightarrow+\infty$. In this direction, we refer the readers to \cite{cheng,secchi}. For existence results involving another classes of potentials, we refer \cite{chang,feng,pala,secchii} and references therein.

There are some papers that have appeared in the recent years regarding the local case of System~\eqref{j000}, which corresponds to the case $s_{1}=s_{2}=1$. In \cite{MR2447960,MR2414222}, A. Ambrosetti et al. considered the following class of linearly coupled systems involving subcritical terms of the form
 $$
  \left\{
  \begin{array}{lr}
  -\Delta u+\mu u=(1+a(x))|u|^{p-1}u+\lambda v, & x\in\mathbb{R}^{N},\\
  -\Delta v+\nu v=(1+b(x))|v|^{q-1}v+\lambda u, & x\in\mathbb{R}^{N}.
  \end{array}
  \right.
 $$
They used concentration compactness type arguments to prove the existence of positive bound and ground states when $\mu=\nu=1$, $\lambda\in(0,1)$, $1<p=q<2^{*}-1$, $a(x)$ and $b(x)$ vanishing at infinity. In \cite{MR2829504}, Z. Chen and W. Zou extended and complemented some results introduced in \cite{MR2414222} by studying the following class of coupled systems
 \begin{equation}\label{s}
  \left\{
  \begin{aligned}
  -\Delta u+\mu u=f_{1}(u)+\lambda v, & \quad x\in\mathbb{R}^{N},\\
  -\Delta v+\nu v=f_{2}(v)+\lambda u, & \quad x\in\mathbb{R}^{N}.
  \end{aligned}
  \right.
 \end{equation}
In \cite{jr1,jr2}, the authors studied the existence of positive ground states for System~\eqref{s} when $N=2$ and $\mu=V_{1}(x)$, $\nu=V_{2}(x)$, $\lambda=\lambda(x)$ are nonnegative functions satisfying suitable assumptions. For more existence results regarding to coupled systems in the local case, we refer the readers to \cite{dc,giovany2,maia2,MR1924476,MR2718702,MR2885948,MR2263573} and references therein. However, there are few works regarding to coupled systems in the nonlocal case, that is, when $s_{1},s_{2}\in(0,1)$. In \cite{guo}, it was studied the following class of coupled systems
  \begin{equation}\label{fra}
 \left\{
 \begin{array}{ll}
 (-\Delta)^{s} u+ u=(|u|^{2p}+b|u|^{p-1}|v|^{p+1})u, & \quad x\in\mathbb{R}^{N},\\
 (-\Delta)^{s} v+ \omega^{2s} v=(|v|^{2p}+b|v|^{p-1}|u|^{p+1})v, & \quad x\in\mathbb{R}^{N},
 \end{array}
 \right.
 \end{equation}
where $\omega>0$, $b>0$ and $2<2p+2<2^{*}_{s}$. By using the Nehari manifold method, the authors proved the existence of ground states for the nonlocal system~\eqref{fra}. Moreover, they proved that if $b>0$ is large enough, then System~\eqref{fra} admits a positive ground state solution. Their results extend and complement the results obtained in \cite{MR2263573} for the local case. In \cite{lu}, it was considered the fractional linearly coupled system \eqref{j000} involving Berestycki-Lions type nonlinearities. In \cite{jr3}, the authors studied System~\eqref{j000} involving nonlinear terms with critical exponential growth of Trudinger-Moser type. We also refer the readers to \cite{choi,fs1} and references therein.


It is important to emphasize that in most of the cited works it was considered the case where the nonlinear term is a powerlike function or a sum of powerlike functions. In this setting was proved several results concerned in existence of solutions with different assumptions on the potentials.  For example, in \cite{MR2263573} was a considered the case where $f_{1}, f_{2}$ are a cubic. Another works considered a more general nonlinear term which satisfies the well known Ambrosetti-Rabinowitz condition at infinity. Namely, for $i=1,2$ there exist $\theta_{i} > 2$ in such way that
\begin{equation}\label{ar}
0 < \theta_{i} F_{i}(t) =\theta_{i}\int_{0}^{t} f_{i}(\tau)\,\ud \tau\leq t f_{i}(t), \quad \mbox{for all} \hspace{0,2cm} t \in \mathbb{R}. \tag{AR}
\end{equation}
Under this condition it follows that any Palais-Smale sequence is bounded. However, there are some superlinear functions $f_{i}$ such that \eqref{ar} is not satisfied, see Remark~\ref{r1}. In the present paper we consider the nonquadraticity condition at infinity introduced by D.G.~Costa and C.A.~Magalh\~aes \cite{MR1306679}. Taking into account the nonquadraticity condition, we are able to prove that any Cerami sequence for the energy functional associated to System~\eqref{j000} is bounded. This is a powerful tool in order to recover the compactness required in variational procedures.



The class of systems considered here imposes several difficulties. The first one is the presence of the fractional laplacian which is a nonlocal operator, that is, it takes care of the behavior of the solution in the whole space. This class of systems is also characterized by its lack of compactness inherent to problem defined on unbounded domains. Here we emphasize that we consider potentials $V_{1}, V_{2}$ that are bounded from below and above by positive constants. Then the loss of compactness provide a serious difficulty in order to guarantee existence of solutions for the System \eqref{j000}. Another obstacle is the fact that the nonlinearities does not verify the well known Ambrosetti-Rabinowitz condition. Moreover, the Schr\"{o}dinger equations are strongly coupled because of the linear terms in the right hand side of System~\eqref{j000}. In order to overcome these difficulties, we apply a fractional version of a result due to P.L. Lion's (see Lemma~\ref{lions}) and we explore the fact that $V_{1},V_{2}$ are periodic or asymptotically periodic. Our approach is variational based on a minimization technique over the Nehari manifold. To our best acknowledgment this is the first work where it is proved the existence of ground states for this class of systems under assumptions involving periodic and asymptotically periodic potentials and nonlinearities which do not satisfy the Ambrosetti-Rabinowitz at infinity.


\subsection{Assumptions and main theorems}

Initially, we deal with the following class of coupled systems
\begin{equation}\label{j0}
	\left\{
	\begin{array}{lr}
		(-\Delta)^{s_{1}} u+V_{1,p}(x)u=f_{1}(u)+\lambda_{p}(x)v, & x\in\mathbb{R}^{N},\\
		(-\Delta)^{s_{2}} v+V_{2,p}(x)v=f_{2}(v)+\lambda_{p}(x)u, & x\in\mathbb{R}^{N},
	\end{array}
	\right. \tag{$S_{\lambda,p}$}
\end{equation}
where $V_{1,p}$, $V_{2,p}$ and $\lambda_{p}$ are $1$-periodic functions for each $x_{1},x_{2},...,x_{N}$. In order to establish a variational approach to treat System~\eqref{j0}, we need to require some suitable assumptions on the potentials $V_{1,p}$ and $V_{2,p}$. For each $i=1,2$, we assume that:

\begin{enumerate}[label=($V_1$),ref=$(V_1)$]
	\item \label{v1}
		$V_{i,p}\in C(\mathbb{R}^{N})$ and there is a constant $V_{p}>0$ such that $V_{i,p}(x)\geq V_{p}$, for all $x\in\mathbb{R}^{N}$.
\end{enumerate}

\begin{enumerate}[label=($V_{2}$),ref=$(V_{2})$]
	\item \label{A3}
	$|\lambda_{p}(x)|\leq\delta\sqrt{V_{1,p}(x)V_{2,p}(x)}$, for some $\delta\in(0,1)$, for all $x\in\mathbb{R}^{N}$.
\end{enumerate}

\noindent Since we are looking for positive solutions we suppose that $f_{i}(s)=0$ for all $s\leq0$. Furthermore, for $i=1,2$, we make the following assumptions on the nonlinearities:

\begin{enumerate}[label=($H_1$),ref=$(H_1)$]
	\item \label{f1}
	  $f_{i}\in C(\mathbb{R})$ and satisfies
	   $$
	    \lim_{t\rightarrow 0^{+}}\frac{f_{i}(t)}{t}=0 \quad \mbox{and} \quad \lim_{t\rightarrow+\infty}\frac{f_{i}(t)}{t}=+\infty.
	   $$
\end{enumerate}

\begin{enumerate}[label=($H_2$),ref=$(H_2)$]
	\item \label{f2}
     There exist $a_{1}>0$ and $p_{i}\in (2,2^{*}_{s_{i}})$ such that
      $$
       |f_{i}(t)|\leq a_{1}(1+t^{p_{i}-1}), \quad \mbox{for all} \hspace{0,2cm} t>0.
      $$
\end{enumerate}
\begin{enumerate}[label=($H_3$),ref=$(H_3)$]
	\item \label{f3}
	There exist $a_{2}>0$ and $\alpha>\frac{N}{2}(p_{0}-2)$ such that
	 $$
	  f_{i}(t)t-2F_{i}(t)\geq a_{2}t^{\alpha}, \quad \mbox{for all} \hspace{0,2cm} t>0,
	 $$
	where $F_{i}(t)=\int_{0}^{t}f_{i}(\tau)\,\ud \tau$ and $p_{0}=\max\{p_{1},p_{2}\}$.
\end{enumerate}

\noindent Now we can state our first result in following form:

\begin{theorem}[Periodic case]\label{A}
	Suppose \ref{v1}, \ref{A3}, \ref{f1}-\ref{f3}. Then System~\eqref{j0} admits at least one nontrivial weak solution. If $\lambda_{p}(x)>0$ for all $x\in\mathbb{R}^{N}$, then
System~\eqref{j0} admits at least one weak solution which is strictly positive.
\end{theorem}

We are also concerned with the existence of solutions for the following class of coupled systems	
\begin{equation}\label{paper1jap}
	\left\{
	\begin{array}{lr}
		(-\Delta)^{s_{1}} u+V_{1}(x)u=f_{1}(u)+\lambda(x)v, & \quad x\in\mathbb{R}^{N},\\
		(-\Delta)^{s_{2}} v+V_{2}(x)v=f_{2}(v)+\lambda(x)u, & x\in\mathbb{R}^{N},
	\end{array}
	\right. \tag{$S_{\lambda}$}
\end{equation}
when the potentials $V_{1},$ $V_{2}$ and $\lambda$ satisfy an asymptotic periodicity condition at infinity. More specifically, for any $\varepsilon>0$, we define the following class of functions
 $$
  \mathcal{F}:=\left\{g\in C(\mathbb{R}^{N})\cap L^{\infty}(\mathbb{R}^{N}):\left|\{x\in\mathbb{R}^{N}:|g(x)|\geq\varepsilon\}\right|<\infty \right\},
 $$
where $|\cdot|$ denotes the Lebesgue measure of a set. For $i=1,2$, we assume the following hypotheses:

\begin{enumerate}[label=($V_3$),ref=$(V_3)$]
	\item \label{v3}
    $V_{i,p}-V_{i}\in\mathcal{F}$ and there is a constant $V_{0}>0$ such that $V_{i,p}(x)> V_{i}(x)\geq V_{0}$, for all $x\in\mathbb{R}^{N}$.
\end{enumerate}

\begin{enumerate}[label=($V_{4}$),ref=$(V_{4})$]
	\item \label{v4}
	$\lambda_{p}-\lambda\in\mathcal{F}$, $\lambda(x)>\lambda_{p}(x)$ and $|\lambda(x)|\leq\delta\sqrt{V_{1}(x)V_{2}(x)}$, for some $\delta\in(0,1)$, for all $x\in\mathbb{R}^{N}$.
\end{enumerate}
\noindent The assumptions \ref{v3} and \ref{v4} imply that $V_{1}$, $V_{2}$ and $\lambda$ are perturbations of periodic functions at infinity. This class of asymptotic periodic functions was introduced by Elves A.B. Silva and Haendel F. Lins in \cite{MR2532816}.

In order to obtain a ground state solution, we also consider the following hypothesis:
\begin{enumerate}[label=($H_4$),ref=$(H_4)$]
	\item \label{f4}
	$\displaystyle\frac{f_{i}(t)}{t}$ is increasing on $(0,+\infty)$.
\end{enumerate}
The assumption \ref{f4} allows us to compare the mountain pass level with the energy level associated with Nehari manifold (see Lemma \ref{nehari}). Under these conditions we are able to state our second main result which can be write in the following form:
\begin{theorem}[Asymptotically periodic case]\label{B}
Suppose \ref{v1}-\ref{v4}, \ref{f1}-\ref{f4}.  Then System~\eqref{paper1jap} admits at least one ground state solution. If $\lambda(x)>0$ for all $x\in\mathbb{R}^{N}$, then the ground state is strictly positive.
\end{theorem}

Finally, we study the behavior of the ground state solutions of System~\eqref{paper1jap} when the coupling function goes to zero. In fact, we prove that the sequence of solutions goes to a positive ground state solution of the uncoupled Schr\"{o}dinger equation. Precisely, we obtain the following result:

\begin{theorem}\label{C}
	Suppose \ref{v1}-\ref{v4}, \ref{f1}-\ref{f4}. Let $(\lambda_{n})_{n} \subset L^{\infty}(\mathbb{R}^{N})$ be a sequence of positive functions such that $\|\lambda_{n}\|_{\infty}\rightarrow0$ as $n\rightarrow+\infty$. For each $n\in\mathbb{N}$, let $(u_{\lambda_{n}},v_{\lambda_n})_{n}$ be a positive ground state solution for System~\eqref{paper1jap} with $\lambda=\lambda_{n}$. Then, up to a subsequence, $(u_{\lambda_{n}},v_{\lambda_n})\rightharpoonup (U_{0},V_{0})$ as $n\rightarrow+\infty$ with one of the following conclusions holding:
	 \begin{itemize}
	 	\item[(i)] $V_{0}\equiv0$ and $U_{0}$ is a positive ground state of
	 	 \[
	 	  (-\Delta)^{s_{1}}u+V_{1}(x)u=f_{1}(u), \quad x\in\mathbb{R}^{N}.
	 	 \]
	 \end{itemize}
     \begin{itemize}
 		\item[(ii)] $U_{0}\equiv0$ and $V_{0}$ is a positive ground state of
 		\[
 		(-\Delta)^{s_{2}}v+V_{2}(x)v=f_{2}(v), \quad x\in\mathbb{R}^{N}.
 		\]
 	 \end{itemize}
\end{theorem}


\begin{remark}\label{r1}
	A typical example of nonlinearity satisfying \ref{f1}-\ref{f4} is given by $f(t) = t ln(1 + |t|)$ for $t>0$ and $f(t)=0$ for $t\leq0$. More generally, we can consider also
$f_{i}(t) = t ln^{\gamma_{i}}(1 + |t|)$ for $t>0$ with $\gamma_{i} \geq 1$ and $f_{i}(t)=0$ for $t\leq0$, $i = 1,2$. In these examples the functions satisfy assumptions \ref{f1}-\ref{f4}. However, these functions do not verify the Ambrosetti-Rabinowtiz condition. In fact, if Ambrosetti-Rabinowitz condition \eqref{ar} is satisfied then we have $F_{i}(t) \geq c_{1} |t|^{\theta_{i}}, t >0$ for each $i = 1,2$ and for some $c_{1} > 0$. In particular, we obtain that $\displaystyle \lim_{t\rightarrow +\infty} F_{i}(t)/|t|^{\theta_{i}} > 0$ holds true. A simple calculation shows that these examples satisfy $\displaystyle \lim_{t\rightarrow +\infty} F_{i}(t)/|t|^{\theta_{i}} = 0$ for any $\theta_{i} > 2$, which implies that \eqref{ar} does not work.
\end{remark}



\subsection{Outline}
The remainder of this paper is organized as follows: In the forthcoming Section, we recall some preliminary concepts about the fractional Laplace operator and we introduce the variational framework to the coupled systems \eqref{j0} and \eqref{paper1jap}. Section \ref{s2} is devoted to the mountain pass geometry to the elliptic system \eqref{j0}. In Section \ref{proofA} we give the proof of Theorem \ref{A}. In order to obtain ground states, in Section \ref{neharimanifold} we introduce and give some properties of Nehari manifold. Moreover, we study the behavior of the ground state energy level. Finally, Sections \ref{proofB} and \ref{proofC} are devoted to the proof of Theorems \ref{B} and \ref{C} respectively.


\subsection{Notation}
Let us introduce the following notation:	
\begin{itemize}
	\item $C$, $\tilde{C}$, $C_{1}$, $C_{2}$,... denote positive constants (possibly different).
	\item $o_{n}(1)$ denotes a sequence which converges to $0$ as $n\rightarrow\infty$;
    \item The norm in $L^{q}(\mathbb{R}^{N})$ and $L^{\infty}(\mathbb{R}^{N})$, will be denoted respectively by $\|\cdot\|_{q}$ and $\|\cdot\|_{\infty}$.
    \item The norm in $L^{q}(\mathbb{R}^{N})\times L^{q}(\mathbb{R}^{N})$ is given by $\|(u,v)\|_{q}=\left(\|u\|^{q}_{q}+\|v\|^{q}_{q}\right)^{1/q}$.

\end{itemize}


\section{Preliminaries and variational framework}\label{s0}

In order to give a variational approach to our problems, we start this section recalling some preliminary concepts about the fractional Laplace operator, for a more complete discussion we refer the readers to \cite{guia}. For $s\in(0,1)$, the \textit{fractional Laplace operator} of a measurable function $u:\mathbb{R}^{N}\rightarrow\mathbb{R}$ is defined by
$$
(-\Delta)^{s}u(x)=-\frac{1}{2}C(N,s)\int_{\mathbb{R}^{N}}\frac{u(x+y)+u(x-y)-2u(x)}{|y|^{1+2s}}\,\ud y,
$$
where
$$
C(N,s)=\left(\int_{\mathbb{R}^{N}}\frac{1-\cos(\xi_{1})}{|\xi|^{N+2s}}\,\ud \xi\right)^{-1}, \quad \xi=(\xi_{1},...,\xi_{N}).
$$
We recall the definition of the fractional Sobolev space
$$
H^{s}(\mathbb{R}^{N})=\left\{u\in L^{2}(\mathbb{R}^{N}):[u]_{s}<\infty\right\},
$$
endowed with the natural norm
$$
\|u\|_{s}=\left([u]_{s}^{2}+\int_{\mathbb{R}^{N}}u^{2}\,\ud x\right)^{1/2}, \quad
[u]_{s}=\left(\int_{\mathbb{R}^{N}}\int_{\mathbb{R}^{N}}\frac{|u(x)-u(y)|^{2}}{|x-y|^{N+2s}}\,\ud x\ud y\right)^{1/2}
$$
where the term $[u]_{s}$ is the so-called \textit{Gagliardo semi-norm} of the function $u$. In light of \cite[Proposition 3.6]{guia} we have that
$$
2C(N,s)^{-1}\|(-\Delta)^{s/2}u\|_{2}^{2}=\int_{\mathbb{R}^{N}}\int_{\mathbb{R}^{N}}\frac{|u(x)-u(y)|^{2}}{|x-y|^{N+2s}}\,\ud x\ud y, \quad \mbox{for all} \hspace{0,2cm} u\in H^{s}(\mathbb{R}^{N}).
$$
For the sake of simplicity, throughout the paper we omit the normalization constants. In view of the presence of the periodic potentials $V_{1,p}$ and $V_{2,p}$ in System~\eqref{j0}, we denote by $E_{i,p}$ the Sobolev space $H^{s_{i}}(\mathbb{R}^{N})$ endowed with the inner product
$$
(u,v)_{E_{i,p}}=\int_{\mathbb{R}^{N}}(-\Delta)^{\frac{s_{i}}{2}} u(-\Delta)^{\frac{s_{i}}{2}} v\,\ud x+\int_{\mathbb{R}^{N}}V_{i,p}(x)uv\,\ud x,
$$
to which corresponds the induced norm $\|u\|_{E_{i,p}}^{2}=(u,u)_{E_{i,p}}$. The fractional critical Sobolev exponent is given by $2^{*}_{s}=2N/(N-2s)$. In light of \cite[Theorem~6.7]{guia} we recall that $E_{i,p}$ is continuously embedded into $L^{q}(\mathbb{R}^{N})$, for $q\in[2,2^{*}_{s_{i}}]$. Here we set the product space $E_{p}=E_{1,p}\times E_{2,p}$ which is a Hilbert space endowed with the natural inner product
$$
((u,v),(w,z))_{E_{p}}=\int_{\mathbb{R}^{N}}\left((-\Delta)^{\frac{s_{1}}{2}} u(-\Delta)^{\frac{s_{1}}{2}} w+V_{1,p}(x)uw+(-\Delta)^{\frac{s_{2}}{2}} v(-\Delta)^{\frac{s_{2}}{2}} z +V_{2,p}(x)vz\right)\,\mathrm{d} x.
$$
We consider the induced norm $\|(u,v)\|_{E_{p}}^{2}=((u,v),(u,v))_{E_{p}}$. Associated to System~\eqref{j0} we have the energy functional $I_{\lambda,p}:E_{p}\rightarrow\mathbb{R}$ given by
$$
I_{\lambda,p}(u,v)=\frac{1}{2}\left(\|(u,v)\|_{E_{p}}^{2}-2\int_{\mathbb{R}^{N}}\lambda_{p}(x)uv\,\mathrm{d} x\right)-\int_{\mathbb{R}^{N}}\left(F_{1}(u)+F_{2}(v)\right)\,\mathrm{d} x.
$$
It follows from assumptions \ref{f1} and \ref{f2} that for any $\varepsilon>0$ there exists $C_{\varepsilon}>0$ such that
\begin{equation}\label{ej1}
F_{i}(t)\leq \varepsilon |t|^{2}+C_{\varepsilon}|t|^{p_{i}}, \quad \mbox{for all} \hspace{0,2cm} t\in\mathbb{R} \hspace{0,2cm} \mbox{and} \hspace{0,2cm} i=1,2.
\end{equation}
Therefore, $I_{\lambda,p}$ is well defined functional on $E_{p}$. Furthermore, we check that $I_{\lambda,p} \in C^1(E_{p},\mathbb{R})$ and
$$
\langle I_{\lambda,p}'(u,v),(\phi,\psi)\rangle  =  ((u,v),(\phi,\psi))_{E_{p}}
-\int_{\mathbb{R}^{N}}\left(f_{1}(u)\phi+f_{2}(v)\psi\right)\,\mathrm{d} x
-\int_{\mathbb{R}^{N}}\lambda_{p}(x)\left(u\psi+v\phi\right)\,\mathrm{d} x.
$$
Hence, critical points of $I_{\lambda,p}$ correspond to weak solutions of System~\eqref{j0} and conversely.

Now we shall consider the elliptic system \eqref{paper1jap}. Taking into account the presence of the bounded potentials $V_{1}$ and $V_{2}$ we denote by $E_{i}$ the Sobolev space $H^{s_{i}}(\mathbb{R}^{N})$ endowed with the inner product
$$
(u,v)_{E_{i}}=\int_{\mathbb{R}^{N}}(-\Delta)^{\frac{s_{i}}{2}} u(-\Delta)^{\frac{s_{i}}{2}} v\,\ud x+\int_{\mathbb{R}^{N}}V_{i}(x)uv\,\ud x,
$$
to which corresponds the induced norm $\|u\|_{E_{i}}^{2}=(u,u)_{E_{i}}$. By the same reason of periodic case, we have that $E_{i}$ is continuously embedded into $L^{q}(\mathbb{R}^{N})$, for $q\in[2,2^{*}_{s_{i}}]$. Here we set the product space $E=E_{1}\times E_{2}$ which is a Hilbert space endowed with the natural inner product
$$
((u,v),(w,z))_{E}=\int_{\mathbb{R}^{N}}\left((-\Delta)^{\frac{s_{1}}{2}} u(-\Delta)^{\frac{s_{1}}{2}} w+V_{1}(x)uw+(-\Delta)^{\frac{s_{2}}{2}} v(-\Delta)^{\frac{s_{2}}{2}} z+V_{2}(x)vz\right)\,\mathrm{d} x.
$$
Moreover, we also consider the induced norm $\|(u,v)\|_{E}^{2}=((u,v),(u,v))_{E}$.

Associated to System~\eqref{paper1jap} we have the energy functional $I_{\lambda}:E \rightarrow\mathbb{R}$ given by
\begin{equation}\label{I}
I_{\lambda}(u,v)=\frac{1}{2}\left(\|(u,v)\|_{E}^{2}-2\int_{\mathbb{R}^{N}}\lambda (x)uv\,\mathrm{d} x\right)-\int_{\mathbb{R}^{N}}\left(F_{1}(u)+F_{2}(v)\right)\,\mathrm{d} x.
\end{equation}
By similar arguments it can be checked that $I_{\lambda} \in C^1(E,\mathbb{R})$ and
$$
\langle I_{\lambda}'(u,v),(\phi,\psi)\rangle  =  ((u,v),(\phi,\psi))_{E}
-\int_{\mathbb{R}^{N}}\left(f_{1}(u)\phi+f_{2}(v)\psi\right)\,\mathrm{d} x
-\int_{\mathbb{R}^{N}}\lambda (x)\left(u\psi+v\phi\right)\,\mathrm{d} x.
$$
Hence, critical points of $I_{\lambda}$ correspond to weak solutions of System~\eqref{paper1jap} and conversely.



\section{Mountain pass geometry}\label{s2}

In this section we give the mountain pass geometry to the energy functional associated to System \eqref{j0}. The same ideas discussed in this section can be applied for the elliptic system \eqref{paper1jap}. It is important to mention that some kind of compactness is required in variational methods. Let X be a Banach space and $I : X \rightarrow \mathbb{R}$ a functional of $C^{1}$ class. It is important to recall that a sequence $(u_{n})_{n} \subset X$ is said to be a Palais-Smale sequence at the level $c \in \mathbb{R}$, whenever $I_{\lambda}(u_{n}) \rightarrow c$ and $\|I^{\prime}(u_{n})\|_{X} \rightarrow 0$ as $n \rightarrow \infty$. Recall also that a sequence $(u_{n})_{n} \subset X$ is said to be a Cerami sequence at the level $c \in \mathbb{R}$, in short $(Ce)_{c}$ sequence, whenever $I_{\lambda}(u_{n}) \rightarrow c$ and $(1 + \|u_{n}\|_{X}) \|I^{\prime}(u_{n})\|_{X^{*}} \rightarrow 0$ as $n \rightarrow \infty$. Since the Ambrosetti-Rabinowitz condition \eqref{ar} it is not available in our setting, we are not able to consider the Palais-Smale condition. In fact, under this condition, we can not verify that any Palais-Smale is bounded. However, by considering the nonquadraticity assumption \ref{f3}, we are able to ensure that any Cerami sequence is bounded. For this purpose, in order to get a nontrivial solution for the fractional coupled systems \eqref{j0} and \eqref{paper1jap}, we shall make use of the following variant of the Mountain Pass Theorem (see \cite{MR1149010}) where it is considered the Cerami condition instead of the Palais-Smale condition.

\begin{theoremletter}\label{MP}
	Let $X$ be a real Banach space with its dual space $X^{*}$, and $J\in C^{1}(X,\mathbb{R})$ be such that
	  \begin{itemize}
	  	\item[($I_{1}$)] there exists $\tau>0$ and $\varrho>0$ such that $J(u)\geq\tau$ provided $\|u\|_{X}=\varrho$;
	  	\item[($I_{2}$)] there exists $e\in X$ with $\|e\|_{X}>\varrho$ such that $J(e)<0$.
	  \end{itemize}
	 Define
	 $$
	  c:=\inf_{\gamma\in\Gamma}\max_{t\in[0,1]}J(\gamma(t)),
	 $$
	where
	 \begin{equation}\label{path}
		 \Gamma=\{\gamma\in C([0,1],X):\gamma(0)=0 \ \mbox{and} \ \gamma(1)=e\}.
     \end{equation}
    Then, there exists a sequence $(u_{n})_{n}\subset X$ such that
	 $$
	  J(u_{n})\rightarrow c \quad \mbox{and} \quad (1+\|u_{n}\|_{X})\|J'(u_{n})\|_{X^{*}}\rightarrow0.
	 $$
\end{theoremletter}

The following Lemma is a consequence of assumption \ref{A3} and will be useful to overcome the difficulty imposed by the coupling function when we study the geometry of the energy functional.

\begin{lemma}\label{nehari-1}
	If \ref{A3} holds, then
	\begin{equation}\label{ej27}
	\|(u,v)\|_{E_{p}}^{2}-2\int_{\mathbb{R}^{N}}\lambda_{p}(x)uv\,\mathrm{d} x\geq(1-\delta)\|(u,v)\|_{E_{p}}^{2}, \quad \mbox{for all} \hspace{0,2cm} (u,v)\in E_{p}.
	\end{equation}
\end{lemma}
\begin{proof}
	In fact, for all $(u,v)\in E_{p}$ we have
	$$
	0\leq\left(\sqrt{V_{1,p}(x)}|u|-\sqrt{V_{2,p}(x)}|v|\right)^{2}=V_{1,p}(x)u^{2}-2\sqrt{V_{1,p}(x)}|u|\sqrt{V_{2,p}(x)}|v|+V_{2,p}(x)v^{2}.
	$$
	Thus, by using assumption \ref{A3} we deduce that
	\[
		-2\int_{\mathbb{R}^{N}}\lambda_{p}(x)uv\,\ud x\geq -\delta\left(\int_{\mathbb{R}^{N}}V_{1,p}(x)u^{2}\,\ud x+\int_{\mathbb{R}^{N}}V_{2,p}(x)v^{2}\,\ud x\right)\geq  -\delta\|(u,v)\|_{E_{p}}^{2},
	\]
	which implies \eqref{ej27}.
\end{proof}	

In the next Lemma we check that $I_{\lambda,p}$ satisfies the mountain pass geometry introduced in Theorem~\ref{MP}.

\begin{lemma}\label{mpg}
	The energy functional $I_{\lambda,p}$ satisfies the mountain pass geometry $(I_{1})$ and $(I_{2})$.
	
\end{lemma}

\begin{proof}	
	Using \eqref{ej1}, \eqref{ej27} and Sobolev embedding we can deduce that
	 \begin{alignat*}{2}
		 I_{\lambda,p}(u,v) & \geq  (1-\delta)\|(u,v)\|_{E_{p}}^{2}-\varepsilon(\|u\|_{2}^{2}+\|v\|_{2}^{2})-C_{\varepsilon}\|u\|_{p_{1}}^{p_{1}}-C_{\varepsilon}\|v\|_{p_{2}}^{p_{2}}\\
		        & \geq  \|(u,v)\|_{E_{p}}^{2}\left(1-\delta-C\varepsilon-C_{\varepsilon}\|(u,v)\|_{E_{p}}^{p_{1}-2}-C_{\varepsilon}\|(u,v)\|_{E_{p}}^{p_{2}-2}\right),
	 \end{alignat*}
	for all $(u,v)\in E_{p}$. Let $\varepsilon>0$ be fixed such that $1-\delta-C\varepsilon>0$. Hence, since $p_{1},p_{2}>2$ we may choose $\varrho>0$ sufficiently small such that $1-\delta-C\varepsilon-C_{\varepsilon}\varrho^{p_{1}-2}-C_{\varepsilon}\varrho^{p_{2}-2}>0$. Therefore, if $\|(u,v)\|_{E_{p}}=\varrho$ then $I_{\lambda,p}(u,v)\geq\tau$, where
	 $$
	  \tau:=\varrho^{2}\left(1-\delta-C\varepsilon-C_{\varepsilon}\varrho^{p_{1}-2}-C_{\varepsilon}\varrho^{p_{2}-2}\right)>0,
	 $$
	which finishes the proof of $(I_{1})$.
	
	In order to prove $(I_{2})$, notice from assumption \ref{f1} that
	 $$
	  \lim_{t\rightarrow+\infty}\frac{F_{i}(t)}{t^{2}}=+\infty, \quad \mbox{for each} \hspace{0,2cm} i=1,2.
	 $$
	Let $\varphi\in C^{\infty}(\mathbb{R}^{N})$, $\varphi>0$ be fixed. Thus, using Fatou's Lemma we have that
	 $$
	  \limsup_{t\rightarrow+\infty}\frac{I_{\lambda,p}(t\varphi,t\varphi)}{t^{2}} \leq \frac{1}{2}\left(\|(\varphi,\varphi)\|_{E_{p}}^{2}-2\int_{\mathbb{R}^{N}}\lambda_{p}(x)\varphi^{2}\,\ud x\right)-\int_{\mathbb{R}^{N}}\liminf_{t\rightarrow+\infty}\frac{F_{1}(t\varphi)+F_{2}(t\varphi)}{(t\varphi)^{2}}\varphi^{2}\,\ud x=-\infty.
	 $$
	Therefore, the result follows considering $(e_{1},e_{2})=(t\varphi,t\varphi)$ for $t$ sufficiently large.
\end{proof}

\begin{remark}
	We emphasize that all results of this section remain true for asymptotically periodic functions proving that $I_{\lambda}$ given in \eqref{I} has the mountain pass geometry.
\end{remark}


\section{Proof of Theorem~\ref{A}}\label{proofA}

As we checked in the preceding section (Lemma~\ref{mpg}), the energy functional $I_{\lambda,p}$ satisfies the mountain pass geometry. Therefore, in view of Theorem~\ref{MP} there exists a $(Ce)_{c}$ sequence $(u_{n},v_{n})_{n}\subset E_{p}$, that is,
 \begin{equation}\label{ps}
  I_{\lambda,p}(u_{n},v_{n})\rightarrow c \quad \mbox{and} \quad (1+\|(u_{n},v_{n})\|_{E_{p}})\|I_{\lambda,p}'(u_{n},v_{n})\|_{E_{p}^{*}}\rightarrow0,
 \end{equation}
where $c$ is the \textit{mountain pass level} introduced in Theorem~\ref{MP}. Notice that we can take a nonnegative Cerami sequence. In fact, let us denote $u_{n}=u_{n}^{+}-u_{n}^{-}$ and $v_{n}=v_{n}^{+}-v_{n}^{-}$, where $u_{n}^{+}:=\max\{u_{n},0\}$, $u_{n}^{-}:=\max\{-u_{n},0\}$, $v_{n}^{+}:=\max\{v_{n},0\}$ and $v_{n}^{-}:=\max\{-v_{n},0\}$. It follows from \ref{v4} that
\[
\int \lambda_{p}(x)(u_{n}v_{n}^{-}+v_{n}u_{n}^-)\,\mathrm{d}x\geq -2\int \lambda_{p}(x)u_{n}^-v_{n}^{-}\,\mathrm{d}x\geq -\delta\|(u_{n}^{-},v_{n}^{-})\|^{2}.
\]
Thus, since $f_{i}(s)=0$ for $s\leq0$ and $i=1,2$, by using \eqref{ps} we conclude that
\[
o_{n}(1)=I_{\lambda,p}^{\prime}(u_{n},v_{n})(-u_{n}^{-},-v_{n}^{-})=\|(u_{n}^{-},v_{n}^{-})\|^{2}+\int \lambda_{p}(x)(u_{n}v_{n}^{-}+v_{n}u_{n}^-)\,\mathrm{d}x\geq (1-\delta)\|(u_{n}^{-},v_{n}^{-})\|^{2},
\]
which implies that $(u_{n}^{-},v_{n}^{-})\rightarrow 0$ strongly $E_{p}$. Therefore, $(u_{n}^{+},v_{n}^{+})_{n}$ is a Cerami sequence. For the sake of simplicity we keep the notation $(u_{n},v_{n})_{n}$.

\begin{proposition}\label{p1}
The sequence $(u_{n},v_{n})_{n}$ given just above is bounded in $E_{p}$.
\end{proposition}
\begin{proof}
First of all, by using assumption \ref{f3} we have
 \begin{eqnarray*}
	 c+o_{n}(1) & = & I_{\lambda,p}(u_{n},v_{n})-\frac{1}{2}\langle I_{\lambda,p}'(u_{n},v_{n}),(u_{n},v_{n})\rangle\\
	            & = & \frac{1}{2}\int_{\mathbb{R}^{N}}(f_{1}(u_{n})u_{n}-2F_{1}(u_{n}))\,\ud x+\frac{1}{2}\int_{\mathbb{R}^{N}}(f_{2}(v_{n})v_{n}-2F_{2}(v_{n}))\,\ud x\\
	            & \geq & \frac{a_{2}}{2}\|(u_{n},v_{n})\|_{\alpha}^{\alpha},
 \end{eqnarray*}
which implies that $\|(u_{n},v_{n})\|_{\alpha}^{\alpha}\leq C$. Now, recall the following interpolation inequality
 $$
  \|u\|_{p}\leq \|u\|_{\alpha}^{t}\|u\|_{\beta}^{1-t}, \quad u\in L^{\alpha}(\mathbb{R}^{N})\cap L^{\beta}(\mathbb{R}^{N}),
 $$
where $0<\alpha\leq p\leq \beta$, $p^{-1}=t\alpha^{-1}+(1-t)\beta^{-1}$ and $t\in[0,1]$. Without any loss of generality we assume that $\alpha <  p_{i}$, for $i=1,2$. Hence, by choosing $\beta=2^{*}_{s_{i}}$ we get
 \begin{align}\label{ii}
  \|u_{n}\|_{p_{1}}^{p_{1}}\leq 2\|u_{n}\|_{\alpha}^{tp_{1}}\|u_{n}\|_{2^{*}_{s_{1}}}^{(1-t)p_{1}} \quad \mbox{and} \quad
  \|v_{n}\|_{p_{2}}^{p_{2}}\leq 2\|v_{n}\|_{\alpha}^{tp_{2}}\|v_{n}\|_{2^{*}_{s_{2}}}^{(1-t)p_{2}}.
 \end{align}
By using \eqref{ej27} one has
 $$
  \frac{1}{2}(1-\delta)\|(u_{n},v_{n})\|_{E_{p}}^{2}\leq \int_{\mathbb{R}^{N}}(F_{1}(u_{n})+F_{2}(v_{n}))\,\ud x+I_{\lambda,p}(u_{n},v_{n}),
 $$
which together with \eqref{ej1}, \eqref{ps} and Sobolev embedding implies that
 $$
  \frac{1}{2}(1-\delta)\|(u_{n},v_{n})\|_{E_{p}}^{2} \leq \varepsilon C\|(u_{n},v_{n})\|_{E_{p}}^{2}+C_{\varepsilon}(\|u_{n}\|_{p_{1}}^{p_{1}}+\|v_{n}\|_{p_{2}}^{p_{2}})+C.
 $$
Taking $\varepsilon>0$ small such that $1-\delta-\varepsilon C>0$ and using \eqref{ii} we deduce that
 \begin{equation}\label{ej2}
  \frac{1}{2}(1-\delta-\varepsilon C)\|(u_{n},v_{n})\|_{E_{p}}^{2}\leq \tilde{C_{\varepsilon}}\|(u_{n},v_{n})\|_{E_{p}}^{(1-t)p_{1}}+\tilde{C_{\varepsilon}}\|(u_{n},v_{n})\|_{E_{p}}^{(1-t)p_{2}}+C.
 \end{equation}
Since $\alpha>\frac{N}{2}(p_{0}-2)$ we conclude that $(1-t)p_{0}<2$. Therefore, \eqref{ej2} implies that $(u_{n},v_{n})_{n}$ is bounded in $E_{p}$. This ends the proof.
\end{proof}

According to Proposition \ref{p1}, we may assume, up to a subsequence, that
\begin{itemize}
	\item $(u_{n},v_{n})\rightharpoonup(u_{0},v_{0})$ weakly in $E_{p}$;
	\item $u_{n}\rightarrow u_{0}$ strongly in $L^{r}_{loc}(\mathbb{R}^{N})$, for all $2\leq r<2^{*}_{s_{1}}$;
	\item $v_{n}\rightarrow v_{0}$ strongly in $L^{s}_{loc}(\mathbb{R}^{N})$, for all $2\leq s<2^{*}_{s_{2}}$;
	\item $u_{n}(x)\rightarrow u_{0}(x)$ and $v_{n}(x)\rightarrow v_{0}(x)$, almost everywhere in $\mathbb{R}^{N}$.
\end{itemize}
Since $C^{\infty}_{0}(\mathbb{R}^{N})\times C^{\infty}_{0}(\mathbb{R}^{N})$ is dense into the space $E_{p}$, it follows by standard arguments that $I_{\lambda,p}'(u_{0},v_{0})=0$, that is, $(u_{0},v_{0})$ is a solution for System~\eqref{j0}.

The next result is important tool to overcome the lack of compactness. The \textit{vanishing lemma} was proved originally by P.L.~Lions \cite[Lemma~I.1]{MR778970} and here we use the following version to fractional Sobolev spaces (see \cite[Lemma~2.4]{secchi}).

\begin{lemma}\label{lions}
	Assume that $(u_{n})_{n}$ is a bounded sequence in $H^{s}(\mathbb{R}^{N})$ satisfying
	\begin{align}\label{paper4lionsa}
	\lim_{n\rightarrow+\infty}\sup_{y\in\mathbb{R}^{N}}\int_{B_{R}(y)}u_{n}^{2}\,\ud x=0,
	\end{align}
	for some $R>0$. Then, $u_{n}\rightarrow0$ strongly in $L^{r}(\mathbb{R}^{N})$, for $2<r<2^{*}_{s}$.
\end{lemma}

In order to get a nontrivial solution, we shall consider the following result:

\begin{proposition}\label{p2}
	Let $(u_{n},v_{n})_{n}\subset E_{p}$ be the $(Ce)_{c}$ sequence satisfying \eqref{ps}. Then, $(u_{n},v_{n})_{n}$ satisfies exactly one of the following conditions:
	 \begin{itemize}
	 	\item[(i)] $(u_{n},v_{n})\rightarrow (0,0)$ strongly in $E_{p}$;
	 	\item[(ii)] There exist a sequence $(y_{n})_{n}\subset\mathbb{R}^{N}$ and constants $R,\eta>0$ such that $|y_{n}|\rightarrow\infty$ as $n\rightarrow\infty$, and
	 	 \begin{equation}\label{vanish}
	 	  \liminf_{n\rightarrow+\infty}\int_{B_{R}(y_{n})}(u_{n}^{2}+v_{n}^{2})\,\ud x\geq\eta>0.
	 	 \end{equation}
	 \end{itemize}
\end{proposition}

\begin{proof}
	Let us suppose that $(ii)$ does not hold. Thus, for any $R>0$ we have
	 $$
	  \lim_{n\rightarrow\infty}\sup_{y\in\mathbb{R}^{N}}\int_{B_{R}(y)}(u_{n}^{2}+v_{n}^{2})\,\mathrm{d} x=0.
	 $$
Hence, it follows from Lemma~\ref{lions} that $u_{n}\rightarrow0$ strongly in $L^{r}(\mathbb{R}^{N})$ for $r\in(2,2^{*}_{s_{1}})$ and $v_{n}\rightarrow0$ strongly in $L^{s}(\mathbb{R}^{N})$ for $s\in(2,2^{*}_{s_{2}})$. Hence, using growth conditions \ref{f1}, \ref{f2}, \eqref{ps} and Lemma~\ref{nehari-1}, we deduce that
 $$
  o_{n}(1)=\langle I_{\lambda,p}'(u_{n},v_{n}),(u_{n},v_{n})\rangle \geq (1-\delta-\varepsilon C)\|(u_{n},v_{n})\|_{E_{p}}^{2}+o_{n}(1).
 $$
Therefore, taking $\varepsilon>0$ small enough such that $1-\delta-\varepsilon C>0$ we conclude that $(i)$ holds. 	
\end{proof}

\begin{proof}[Proof of Theorem~\ref{A} completed]
		If $(u_{0},v_{0})\neq(0,0)$, then we already have a nontrivial solution for System~\eqref{j0}. If $(u_{0},v_{0})=(0,0)$, since $I_{\lambda,p}(u_{n},v_{n})\rightarrow c>0$ and $I_{\lambda,p}$ is continuous, it follows that $(u_{n},v_{n})_{n}$ can not go to zero strongly in $E_{p}$. Thus, from Proposition~\ref{p2}, we obtain a sequence $(y_{n})_{n}\subset\mathbb{R}^{N}$ and constants $R,\eta>0$ such that
		 \begin{equation}\label{ej4}
		  \liminf_{n\rightarrow+\infty}\int_{B_{R}(y_{n})}(u_{n}^{2}+v_{n}^{2})\,\ud x\geq\eta>0.
		 \end{equation}
		Let us consider the shift sequence $(\tilde{u}_{n}(x),\tilde{v}_{n}(x))=(u_{n}(x+y_{n}),v_{n}(x+y_{n}))$. Since $V_{1,p}(\cdot)$, $V_{2,p}(\cdot)$ and $\lambda_{p}(\cdot)$ are periodic, it follows that the energy functional $I_{\lambda,p}$ is invariant by translations of the form $(u,v)\mapsto (u(\cdot-z),v(\cdot-z))$ with $z\in\mathbb{Z}^{N}$. By a standard computation we can deduce that
		$$
		\|(u_{n},v_{n})\|_{E_{p}}=\|(\tilde{u}_{n},\tilde{v}_{n})\|_{E_{p}} \quad \mbox{and} \quad I_{\lambda,p}(u_{n},v_{n})=I_{\lambda,p}(\tilde{u}_{n},\tilde{v}_{n})\rightarrow c.
		$$
Furthermore, we also have
$$
(1+\|(\tilde{u}_{n},\tilde{v}_{n})\|_{E_{p}})\|I_{\lambda,p}'(\tilde{u}_{n},\tilde{v}_{n})\|_{E_{p}^{*}}\rightarrow0.
$$
Moreover, arguing as in the proof of Proposition~\ref{p1} we can conclude that $(\tilde{u}_{n},\tilde{v}_{n})_{n}$ is a bounded sequence in $E_{p}$. Thus, up to a subsequence, $(\tilde{u}_{n},\tilde{v}_{n})\rightharpoonup(\tilde{u},\tilde{v})$ weakly in $E_{p}$ and $(\tilde{u}_{n},\tilde{v}_{n})\rightarrow(\tilde{u},\tilde{v})$ strongly in $L^{2}(B_{R}(0))\times L^{2}(B_{R}(0))$. Moreover, $(\tilde{u},\tilde{v})$ is a critical point of $I_{\lambda,p}$. Using \eqref{ej4} we obtain
		$$
		\int_{B_{R}(0)}(\tilde{u}^{2}+\tilde{v}^{2})\,\ud x=\liminf_{n\rightarrow\infty}\int_{B_{R}(0)}(\tilde{u}_{n}^{2}+\tilde{v}_{n}^{2})\,\ud x=\liminf_{n\rightarrow\infty}\int_{B_{R}(y_{n})}(u_{n}^{2}+v_{n}^{2})\,\ud x\geq\eta>0.
		$$
Therefore, $(\tilde{u},\tilde{v})$ is a nontrivial weak solution for System~\eqref{j0}.

Finally, let us prove that if $\lambda_{p}(x)>0$ for all $x\in\mathbb{R}^{N}$, then the weak solution is positive. First, let us prove that $\tilde{u}\not\equiv0$ and $\tilde{v}\not\equiv0$. Suppose without loss of generality that $\tilde{u}\not\equiv0$. If $\tilde{v}\equiv0$, then
\[
0=\langle I_{\lambda,p}^{\prime}(\tilde{u},\tilde{v}),(0,\psi)\rangle=-\int_{\mathbb{R}^{N}}\lambda_{p}(x)\tilde{u}\psi\,\mathrm{d}x, \quad \mbox{for all} \hspace{0,2cm} \psi\in C^{\infty}_{0}(\mathbb{R}^{N}).
\]
Since $\lambda_{p}(x)>0$ for all $x\in \mathbb{R}^{N}$ we have that $\tilde{u}\equiv0$ which is a contradiction. Therefore, $\tilde{v}\not\equiv0$. Let us denote $\tilde{u}=\tilde{u}^{+}-\tilde{u}^{-}$ and $\tilde{v}=\tilde{v}^{+}-\tilde{v}^{-}$, where $\tilde{u}^{+}:=\max\{\tilde{u},0\}$, $\tilde{u}^{-}:=\max\{-\tilde{u},0\}$, $\tilde{v}^{+}:=\max\{\tilde{v},0\}$ and $\tilde{v}^{-}:=\max\{-\tilde{v},0\}$. It follows from \ref{v4} that
\[
\int_{\mathbb{R}^{N}} \lambda_{p}(x)(\tilde{u}\tilde{v}^{-}+\tilde{v}\tilde{u}^-)\,\mathrm{d}x\geq -2\int_{\mathbb{R}^{N}} \lambda_{p}(x)\tilde{u}^-\tilde{v}^{-}\,\mathrm{d}x\geq -\delta\|(\tilde{u}^{-},\tilde{v}^{-})\|^{2}.
\]
Thus, since $f_{i}(s)=0$ for $s\leq0$ and $i=1,2$, we have that
\[
0=I_{\lambda,p}^{\prime}(\tilde{u},\tilde{v})(-\tilde{u}^{-},-\tilde{v}^{-})=\|(\tilde{u}^{-},\tilde{v}^{-})\|^{2}+\int_{\mathbb{R}^{N}} \lambda_{p}(x)(\tilde{u}\tilde{v}^{-}+\tilde{v}\tilde{u}^-)\,\mathrm{d}x\geq (1-\delta)\|(\tilde{u}^{-},\tilde{v}^{-})\|^{2},
\]
which implies that $\|(\tilde{u}^{-},\tilde{v}^{-})\|^{2}= 0$. Therefore, $(\tilde{u}^{-},\tilde{v}^{-})=(0,0)$ and $(\tilde{u},\tilde{v})=(\tilde{u}^{+},\tilde{v}^{+})$ is a nonnegative solution for System~\eqref{j0}. By using Strong Maximum Principle in each equation of System~\eqref{j0}, we conclude that $\tilde{u}$ and $\tilde{v}$ are positive which finishes the proof of Theorem~\ref{A}.
\end{proof}



\section{The Nehari manifold}\label{neharimanifold}

In order to get a ground state solution, we introduce the Nehari manifolds associated to Systems~\eqref{j0} and \eqref{paper1jap} respectively defined by
 \[
 \mathcal{N}_{\lambda,p}:=\{(u,v)\in E_{p}\backslash\{(0,0)\}:\langle I_{\lambda,p}'(u,v),(u,v)\rangle=0\},
 \]
 \[
  \mathcal{N}_{\lambda}:=\{(u,v)\in E\backslash\{(0,0)\}:\langle I_{\lambda}'(u,v),(u,v)\rangle=0\}.
 \]
Since $f_{i}(t)=0$ for all $t\leq0$ and each $i=1,2$, it is not hard to check that if $(u,v)\in\mathcal{N}_{\lambda,p},\mathcal{N}_{\lambda}$, then $|\{u>0\}|>0$ or $|\{v>0\}|>0$. Let us define the set
 \[
  E_{+}=\{ (u,v)\in E \backslash \{(0,0)\}: |\{u>0\}|>0 \ or \ |\{v>0\}|>0 \}.
 \]
By similar ideas to \cite{jr1,jr2} we can obtain the following Lemma:
\begin{lemma}\label{proj}
	For any $(u,v)\in E_{+}$, there exists a unique $t_{0}>0$, depending on $(u,v)$ and $\lambda$, such~that
	\[
	(t_{0}u,t_{0}v)\in\mathcal{N}_{\lambda} \quad \mbox{and} \quad  I_{\lambda}(t_{0}u,t_{0}v)=\max_{t\geq0} I_{\lambda}(tu,tv).
	\]
\end{lemma}


\begin{lemma}\label{gs}
	If \ref{f4} holds, then	$f_{i}(t)s-2F_{i}(t)$ is increasing for~$t>0$ and $i=1,2.$
\end{lemma}
\begin{proof}
	In fact, let $0<t_{1}<t_{2}$ be fixed. Using \ref{f4} we deduce that
	\begin{equation}\label{gs1}
	f_{i}(t_{1})t_{1}-2F_{i}(t_{1}) < \frac{f_{i}(t_{2})}{t_{2}}t_{1}^{2}-2F_{i}(t_{2})+2\int_{t_{1}}^{t_{2}}f_{i}(\tau)\,\ud\tau.
	\end{equation}
	Moreover,
	\begin{equation}\label{gs2}
	2\int_{t_{1}}^{t_{2}}f_{i}(\tau)\,\ud\tau<2\frac{f_{i}(t_{2})}{t_{2}}\int_{t_{1}}^{t_{2}}\tau\,\ud\tau=\frac{f_{i}(t_{2})}{t_{2}}(t_{2}^{2}-t_{1}^{2}).
	\end{equation}
	Combining \eqref{gs1} and \eqref{gs2} we conclude that
	$$
	f_{i}(t_{1})t_{1}-2F_{i}(t_{1}) < f_{i}(t_{2})t_{2}-2F_{i}(t_{2}),
	$$
	which finishes the proof.
\end{proof}

We introduce the Nehari energy levels associated with Systems~\eqref{j0} and \eqref{paper1jap} respectively by
\[
c_{\mathcal{N}_{\lambda,p}}=\inf_{(u,v)\in\mathcal{N}_{\lambda,p}} I_{\lambda,p}(u,v) \quad \mbox{and} \quad c_{\mathcal{N}_{\lambda}}=\inf_{(u,v)\in\mathcal{N}_{\lambda}} I_{\lambda}(u,v).
\]
By using standard arguments it is not hard to check that under our assumptions the levels $c_{\mathcal{N}_{\lambda,p}}$ and $c_{\mathcal{N}_{\lambda}}$ are positive for all nonnegative coupling function $\lambda$. The remainder of this section is devoted to study the behavior of $c_{\mathcal{N}_{\lambda,p}}$ and $c_{\mathcal{N}_{\lambda}}$. The next Lemma establish some estimates in order to compare the mountain pass level and the least energy level.

\begin{lemma}\label{nehari}
	The following estimates hold:
	\begin{itemize}
		\item[(i)] $c\leq c_{\mathcal{N}_{\lambda}}$;
		\item[(ii)] $c_{\mathcal{N}_{\lambda}}<c_{\mathcal{N}_{\lambda,p}}$.
	\end{itemize}
\end{lemma}
\begin{proof}
	Let $(u,v)\in\mathcal{N}_{\lambda}$ be fixed. In view of Lemma~\ref{proj} we have that $I_{\lambda}(u,v)=\max_{t\geq0}I_{\lambda}(tu,tv)$. Let $\gamma:[0,1]\rightarrow E$ be defined by $\gamma(t)=(tt_{0}u,tt_{0}v)$, where $t_{0}>0$ large enough such that $I_{\lambda}(t_{0}u,t_{0}v)<0$. Thus, $\gamma\in\Gamma$. Therefore,
	\begin{equation}\label{ej10}
	c\leq\max_{t\in[0,1]}I_{\lambda}(\gamma(t))\leq\max_{t\geq0}I_{\lambda}(tu,tv)=I_{\lambda}(u,v).
	\end{equation}
	Since \eqref{ej10} holds for all $(u,v)\in\mathcal{N}_{\lambda}$, we conclude that $c\leq c_{\mathcal{N}_{\lambda}}$.
	
	In order to prove $(ii)$, let $(u_{n},v_{n})_{n}\subset\mathcal{N}_{\lambda,p}$ be a minimizing sequence for $c_{\mathcal{N}_{\lambda,p}}$, that is, $I_{\lambda,p}(u_{n},v_{n})\rightarrow c_{\mathcal{N}_{\lambda,p}}$. It is well known that under our assumptions $\mathcal{N}_{\lambda,p}$ is a natural constraint to our problem, that is, critical points of $I_{\lambda,p}\mid_{\mathcal{N}_{\lambda,p}}$ are critical points of $I_{\lambda,p}$. This is a consequence of Lagrange multiplier Theorem. Hence, similarly to Section~\ref{proofA}, we are able to prove that, up to a subsequence, $(u_{n},v_{n})\rightharpoonup (u,v)$ weakly in $E_{p}$, where $u>0$, $v>0$ and $I_{\lambda,p}^{\prime}(u,v)=0$. Obviously, $c_{\mathcal{N}_{\lambda,p}}\leq I_{\lambda,p}(u,v)$. On the other hand, in view of Lemma~\ref{gs} and Fatou's Lemma, we deduce that
	 \begin{alignat*}{2}
	   c_{\mathcal{N}_{\lambda,p}}+o_{n}(1) & = I_{\lambda,p}(u_{n},v_{n})-\frac{1}{2}\langle I_{\lambda,p}^{\prime}(u_{n},v_{n}),(u_{n},v_{n})\rangle\\
	                            & = \frac{1}{2}\int_{\mathbb{R}^{N}}(f_{1}(u_{n})u_{n}-2F_{1}(u_{n}))\,\mathrm{d}x+\frac{1}{2}\int_{\mathbb{R}^{N}}(f_{2}(v_{n})v_{n}-2F_{2}(v_{n}))\,\mathrm{d} x\\
	                            & \geq \frac{1}{2}\int_{\mathbb{R}^{N}}(f_{1}(u)u-2F_{1}(u))\,\mathrm{d} x+\frac{1}{2}\int_{\mathbb{R}^{N}}(f_{2}(v)v-2F_{2}(v))\,\mathrm{d} x+o_{n}(1)\\
	                            & = I_{\lambda,p}(u,v)-\frac{1}{2}\langle I_{\lambda,p}^{\prime}(u,v),(u,v)\rangle+o_{n}(1)\\
	                            & = I_{\lambda,p}(u,v)+o_{n}(1),
	 \end{alignat*}
	which implies that $c_{\mathcal{N}_{\lambda,p}}\geq I_{\lambda,p}(u,v)$. Therefore, $I_{\lambda,p}(u,v)=c_{\mathcal{N}_{\lambda,p}}$. In light of \ref{v4} one has
	\[
	\int_{\mathbb{R}^{N}}\left\{[V_{1}(x)-V_{1,p}(x)]u^{2}+[V_{2}(x)-V_{2,p}(x)]v^{2}+[\lambda_{p}(x)-\lambda(x)]uv \right\}\,\mathrm{d}x<0.
	\]
	In view of Lemma~\ref{proj}, there exists a unique $t_{0}>0$ such that $(t_{0}u,t_{0}v)\in\mathcal{N}_{\lambda}$. Hence, it follows that $I_{\lambda}(t_{0}u,t_{0}v)-I_{\lambda,p}(t_{0}u,t_{0}v)<0$. Therefore, we have
	\[
	c_{\mathcal{N}_{\lambda}}\leq I_{\lambda}(t_{0}u,t_{0}v)<I_{\lambda,p}(t_{0}u,t_{0}v)\leq \max_{t\geq0}I_{\lambda,p}(tu,tv)=I_{\lambda,p}(u,v)=c_{\mathcal{N}_{\lambda,p}},
	\]
	which implies $(ii)$ and finishes the proof.
\end{proof}

\begin{proposition}
	The map $\lambda \mapsto c_{\mathcal{N}_{\lambda}}$ is decreasing in the following sense: if $\lambda_{1},\lambda_{2}\in L^{\infty}(\mathbb{R}^{N})$ satisfy $\lambda_{1}(x)<\lambda_{2}(x)$ for all $x\in\mathbb{R}^{N}$, then $c_{\mathcal{N}_{\lambda_{2}}}<c_{\mathcal{N}_{\lambda_{1}}}$.
\end{proposition}
\begin{proof}
	Let $(u,v)\in\mathcal{N}_{\lambda_{1}}$ be such that $I_{\lambda_{1}}(u,v)=c_{\mathcal{N}_{\lambda_{1}}}$ (see Section~\ref{proofB}). In view of Lemma~\ref{proj}, there exists $t_{0}>0$ such that $(t_{0}u,t_{0}v)\in\mathcal{N}_{\lambda_{2}}$. Notice that $I_{\lambda_{2}}(t_{0}u,t_{0}v)<I_{\lambda_{1}}(t_{0}u,t_{0}v)$. Thus, we have
	\[
	c_{\mathcal{N}_{\lambda_{2}}}\leq I_{\lambda_{2}}(t_{0}u,t_{0}v)<I_{\lambda_{1}}(t_{0}u,t_{0}v)\leq \max_{t\geq0}I_{\lambda_{1}}(tu,tv)=I_{\lambda_{1}}(u,v)=c_{\mathcal{N}_{\lambda_{1}}}.
	\]
	Therefore, $c_{\mathcal{N}_{\lambda_{2}}}<c_{\mathcal{N}_{\lambda_{1}}}$ and the map $\lambda \mapsto c_{\mathcal{N}_{\lambda}}$ is decreasing.
\end{proof}

\begin{proposition}\label{conve}
	Let $(\lambda_{n})_{n}\subset L^{\infty}(\mathbb{R}^{N})$ be a sequence of positive functions such that $\lambda_{n}\rightarrow\lambda$ strongly in $L^{\infty}(\mathbb{R}^{N})$. Then, one has
	\[
	\displaystyle\lim_{n\rightarrow+\infty}c_{\mathcal{N}_{\lambda_{n}}}=c_{\mathcal{N}_{\lambda}}.
	\]
\end{proposition}
\begin{proof}
	Let $(u_{\lambda},v_{\lambda})\in\mathcal{N}_{\lambda}$ be a positive ground state solution for System~\eqref{paper1jap} (see Section~\ref{proofB}). For each $n\in\mathbb{N}$, there exists $t_{n}>0$ such that $(t_{n}u_{\lambda},t_{n}v_{\lambda})\in\mathcal{N}_{\lambda_{n}}$. Thus, taking account that $\lambda_{n}\rightarrow\lambda$ strongly in $L^{\infty}(\mathbb{R}^{N})$ we can deduce~that
	\begin{equation*}
	o_{n}(1)= \int_{\mathbb{R}^{N}}\left(\frac{f_{1}(t_{n}u_{\lambda})}{t_{n}}u_{\lambda}-f_{1}(u_{\lambda})u_{\lambda} \right)\,\mathrm{d}x+\int_{\mathbb{R}^{N}}\left(\frac{f_{2}(t_{n}v_{\lambda})}{t_{n}}v_{\lambda}-f_{2}(v_{\lambda})v_{\lambda} \right)\,\mathrm{d}x.
	\end{equation*}
	We claim that $\lim_{n\rightarrow+\infty}t_{n}= 1$. In fact, arguing by contradiction let us suppose that $\limsup_{n\rightarrow+\infty}t_{n}> 1$. Hence, $t_{n}\geq 1+\varepsilon_{0}$ for $n\in\mathbb{N}$ large. By using \ref{f4} we obtain
	\begin{equation*}
	o_{n}(1)\geq \int_{\mathbb{R}^{N}}\left(\frac{f_{1}((1+\varepsilon_{0})u_{\lambda})}{1+\varepsilon_{0}}u_{\lambda}-f_{1}(u_{\lambda})u_{\lambda} \right)\,\mathrm{d}x+\int_{\mathbb{R}^{N}}\left(\frac{f_{2}((1+\varepsilon_{0})v_{\lambda})}{1+\varepsilon_{0}}v_{\lambda}-f_{2}(v_{\lambda})v_{\lambda} \right)\,\mathrm{d}x.
	\end{equation*}
	Thus, it follows that
	\begin{equation*}
	o_{n}(1)> \int_{\mathbb{R}^{N}}\left(f_{1}(u_{\lambda})u_{\lambda}-f_{1}(u_{\lambda})u_{\lambda} \right)\,\mathrm{d}x+\int_{\mathbb{R}^{N}}\left(f_{2}(v_{\lambda})v_{\lambda}-f_{2}(v_{\lambda})v_{\lambda} \right)\,\mathrm{d}x=0,
	\end{equation*}
	which is not possible. Thus, we have concluded that $\limsup_{n\rightarrow+\infty}t_{n}\leq 1$. If we suppose that $\limsup_{n\rightarrow+\infty}t_{n}< 1$, we get a contradiction applying similar arguments. Hence, $\limsup_{n\rightarrow+\infty}t_{n}= 1$. Analogously, we can check that $\liminf_{n\rightarrow+\infty}t_{n}= 1$. Therefore, $\lim_{n\rightarrow+\infty}t_{n}= 1$. Finally, since $(t_{n}u_{\lambda},t_{n}v_{\lambda})\in\mathcal{N}_{\lambda_{n}}$, we have that
	\begin{alignat*}{2}
	c_{\mathcal{N}_{\lambda_{n}}} & \leq I_{\lambda_{n}}(t_{n}u_{\lambda},t_{n}v_{\lambda})\\
	& = \frac{t_{n}^{2}}{2}\left(\|(u_{\lambda},v_{\lambda})\|_{E}-2\int_{\mathbb{R}^{N}}\lambda_{n}(x)u_{\lambda}v_{\lambda}\,\mathrm{d}x \right)-\int_{\mathbb{R}^{N}}(F_{1}(t_{n}u_{\lambda})+F_{2}(t_{n}v_{\lambda}))\,\mathrm{d}x\\
	& = \frac{1}{2}\left(\|(u_{\lambda},v_{\lambda})\|_{E}-2\int_{\mathbb{R}^{N}}\lambda(x)u_{\lambda}v_{\lambda}\,\mathrm{d}x \right)-\int_{\mathbb{R}^{N}}(F_{1}(u_{\lambda})+F_{2}(v_{\lambda}))\,\mathrm{d}x+o_{n}(1)\\
	& = I_{\lambda}(u_{\lambda},v_{\lambda})+o_{n}(1),
	\end{alignat*}
	which implies that
	$
	\lim_{n\rightarrow+\infty}c_{\mathcal{N}_{\lambda_{n}}}\leq c_{\mathcal{N}_{\lambda}}.
	$
	On the other hand, for each $n\in\mathbb{N}$ let $(u_{n},v_{n})_{n}\in\mathcal{N}_{\lambda_{n}}$ be such that $I_{\lambda_{n}}(u_{n},v_{n})=c_{\mathcal{N}_{\lambda_{n}}}$. Notice that
	\begin{alignat*}{2}
	c_{\mathcal{N}_{\lambda}} +o_{n}(1) & = I_{\lambda_{n}}(u_{n},v_{n})-\frac{1}{2}\langle I_{\lambda_{n}}^{\prime}(u_{n},v_{n}),(u_{n},v_{n})\rangle\\
	& = \frac{1}{2}\int_{\mathbb{R}^{N}}(f_{1}(u_{n})u_{n}-2F_{1}(u_{n}))\,\ud x+\frac{1}{2}\int_{\mathbb{R}^{N}}(f_{2}(v_{n})v_{n}-2F_{2}(v_{n}))\,\ud x\\
	& \geq  \frac{a_{2}}{2}\|(u_{n},v_{n})\|_{\alpha}^{\alpha},
	\end{alignat*}
	which implies that $\|(u_{n},v_{n})\|_{\alpha}^{\alpha}\leq C$. Arguing as in the proof of Proposition~\ref{p1}, we can conclude that $(u_{n},v_{n})_{n}$ is bounded in $E$. In view of Lemma~\ref{proj}, there exists a sequence $(t_{n})_{n}\subset(0,+\infty)$ such that $(t_{n}u_{n},t_{n}v_{n})_{n}\subset\mathcal{N}_{\lambda}$. Arguing as before, it is not hard to check that $\lim_{n\rightarrow+\infty}t_{n}=1$. Therefore, we have
	\begin{alignat*}{2}
	c_{\mathcal{N}_{\lambda}} & \leq I_{\lambda}(t_{n}u_{n},t_{n}v_{n})\\
	& = \frac{t_{n}^{2}}{2}\left(\|(u_{n},v_{n})\|_{E}-2\int_{\mathbb{R}^{N}}\lambda(x)u_{n}v_{n}\,\mathrm{d}x \right)-\int_{\mathbb{R}^{N}}(F_{1}(t_{n}u_{n})+F_{2}(t_{n}v_{n}))\,\mathrm{d}x\\
	& = I_{\lambda_{n}}(t_{n}u_{n},t_{n}v_{n})+t_{n}^{2}\int_{\mathbb{R}^{N}}(\lambda_{n}(x)-\lambda(x))u_{n}v_{n}\,\mathrm{d}x\\
	& \leq \max_{t\geq0}I_{\lambda_{n}}(tu_{n},tv_{n})+o_{n}(1)\\
	& =I_{\lambda_{n}}(u_{n},v_{n})+o_{n}(1),
	\end{alignat*}
	which implies that $c_{\mathcal{N}_{\lambda}}\leq \lim_{n\rightarrow+\infty}c_{\mathcal{N}_{\lambda_{n}}}$ and finishes the proof.
\end{proof}


\section{Proof of Theorem~\ref{B}}\label{proofB}

In this Section we study the existence of ground states for the class of linearly coupled systems \eqref{paper1jap}. For this purpose, we introduce the following sets:
 $$
 I^{b}:=\{(u,v)\in E:I_{\lambda}(u,v)\leq b\},
 $$
  \vspace{-0,7cm}
 $$
  K:=\{(u,v)\in E:I_{\lambda}'(u,v)=0\},
 $$
  \vspace{-0,7cm}
 $$
  K_{b}:=\{(u,v)\in K: I_{\lambda}(u,v)=b\}.
 $$
In order to get a nontrivial solution for \eqref{paper1jap}, we can not repeat the idea used in Section~\ref{proofA}, since the energy function $I$ is not invariant by translations. In order to overcome this difficulty, we shall use the following local version of the Mountain Pass Theorem (see \cite{MR2532816}):

\begin{theoremletter}\label{elves}
	Let $X$ be a real Banach space. Suppose that $J\in C^{1}(X,\mathbb{R})$ satisfies $J(0)=0$, $(I_{1})$ and $(I_{2})$ (see Theorem~\ref{MP}). If there exists $\gamma_{0}\in\Gamma$, $\Gamma$ defined by \eqref{path}, such that
	 $$
	  c=\inf_{\gamma\in\Gamma}\max_{t\in[0,1]}J(\gamma(t))=\max_{t\in[0,1]}J(\gamma_{0}(t))>0,
	 $$
	then $J$ possesses a nontrivial critical point $u\in K_{c}\cap\gamma_{0}([0,1])$.
\end{theoremletter}

For any $\varepsilon>0$, $R>0$ and $h\in\mathcal{F}$ we set $D_{\varepsilon}(R):=\{x\in\mathbb{R}^{N}:|h(x)|\geq\varepsilon\}$. In \cite{MR2532816}, the authors proved the following lemma:

\begin{lemma}\label{D}
	For $h\in\mathcal{F}$ it follows that $|D_{\varepsilon}(R)|\rightarrow0$, as $R\rightarrow\infty$.
\end{lemma}

In order to get a nontrivial solution for System~\eqref{paper1jap}, we prove the following technical lemma:

\begin{lemma}\label{conv}
	Let $(u_{n},v_{n})_{n}\subset E$ be a bounded sequence and $(\varphi_{n}(x),\psi_{n}(x))=(\varphi(x-y_{n}),\psi(x-y_{n}))$, where $(\varphi,\psi)\in C^{\infty}_{0}(\mathbb{R}^{N})\times C^{\infty}_{0}(\mathbb{R}^{N})$, where $(y_{n})_{n}\subset\mathbb{R}^{N}$ such that $|y_{n}|\rightarrow\infty$, as $n\rightarrow\infty$. Then, we have the following convergences
	 \begin{align}
	  (V_{1,p}(x)-V_{1}(x))u_{n}\varphi_{n}\rightarrow0,\label{ej11}\\	
	  (V_{2,p}(x)-V_{2}(x))v_{n}\psi_{n}\rightarrow0,\label{ej12}\\	
	  (\lambda_{p}(x)-\lambda(x))(u_{n}\psi_{n}+v_{n}\varphi_{n})\rightarrow0\label{ej13},
	 \end{align}
	strongly in $L^{1}(\mathbb{R}^{N})$, as $n\rightarrow\infty$.
\end{lemma}
\begin{proof}
	The proof is quite similar to \cite[Lemma~5.1]{MR2532816} and for reader's convenience we sketch the proof here. Let us consider the proof for \eqref{ej11}. It is well known that given $\varphi\in L^{2}(\mathbb{R}^{N})$ and $\delta>0$, there exists $\varepsilon\in(0,\delta)$ such that for every measurable set $A\subset\mathbb{R}^{N}$ satisfying $|A|<\varepsilon$, we have
	 \begin{equation}\label{ej14}
	  \int_{A}\varphi\,\ud x<\delta.
	 \end{equation}	
	It follows from Lemma~\ref{D} that for any $\varepsilon>0$, there exists $R>0$ such that $|D_{\varepsilon}(R)|<\varepsilon$. By using H\"{o}lder inequality we deduce that
	 \begin{eqnarray*}
	  \int_{\mathbb{R}^{N}\backslash B_{R}(0)}|V_{1,p}(x)-V_{1}(x)||u_{n}||\varphi_{n}|\,\ud x & \leq & 2 \|V_{1,p}\|_{\infty} \int_{(\mathbb{R}^{N}\backslash B_{R}(0)) \cap D_{\varepsilon}(R)}|u_{n}||\varphi_{n}|\,\ud x \\
&+& \varepsilon \int_{(\mathbb{R}^{N}\backslash B_{R}(0)) \cap D_{\varepsilon}(R)^{c}}|u_{n}||\varphi_{n}|\,\ud x\\
& \leq & 2 \|V_{1,p}\|_{\infty}\|u_{n}\|_{L^{2}(D_{\varepsilon}(R))}\|\varphi_{n}\|_{L^{2}(D_{\varepsilon}(R))}+\varepsilon\|u_{n}\|_{2}\|\varphi\|_{2},
	 \end{eqnarray*}
	which together with \eqref{ej14} and the fact that $(u_{n},v_{n})_{n}$ is bounded in $E$ implies that
	 \begin{equation}\label{ej15}
		 \int_{\mathbb{R}^{N}\backslash B_{R}(0)}|V_{1,p}(x)-V_{1}(x)||u_{n}||\varphi_{n}|\,\ud x\leq C_{1}(\delta^{1/2}+\delta).
	 \end{equation}	
	On the other hand, using the fact that $\varphi\in L^{2}(\mathbb{R}^{N})$ and $|y_{n}|\rightarrow\infty$, we obtain $n_{0}\in\mathbb{N}$ such that
	 \begin{equation}\label{ej16}
		  \int_{B_{R}(0)}|V_{1,p}(x)-V_{1}(x)||u_{n}||\varphi_{n}|\,\ud x\leq C_{2}\|\varphi\|_{L^{2}(B_{R}(x-y_{n}))}\leq C_{2}\delta, \quad \mbox{for all} \hspace{0,2cm} n\geq n_{0}.
	 \end{equation}
	Since $\delta>0$ is arbitrary, the inequalities \eqref{ej15} and \eqref{ej16} imply \eqref{ej11}. The convergences \eqref{ej12} and \eqref{ej13} follow by a similar argument.
\end{proof}

\begin{proof}[Proof of Theorem~\ref{B} completed] It is easy to see that Lemma~\ref{mpg} remains true for the energy functional $I$, that is, $I$ satisfies the mountain pass geometry. Thus, it follows from Theorem \ref{MP} that there exists a nonnegative $(Ce)_{c}$ sequence $(u_{n},v_{n})_{n}\subset E$, that is,
 \begin{equation}\label{css}
  I_{\lambda}(u_{n},v_{n})\rightarrow c \quad \mbox{and} \quad (1+\|(u_{n},v_{n})\|_{E})\|I_{\lambda}'(u_{n},v_{n})\|_{E^{*}}\rightarrow0.
 \end{equation}
By the same ideas used in Proposition~\ref{p1} we can conclude that $(u_{n},v_{n})_{n}$ is bounded in $E$. Thus, we may assume, up to a subsequence, that $(u_{n},v_{n})\rightharpoonup (u_{0},v_{0})$ weakly in $E$. Hence, by using a density argument, we can deduce that $I_{\lambda}'(u_{0},v_{0})=0$. If $(u_{0},v_{0})\neq(0,0)$, then we are done. Now, let us suppose that $(u_{0},v_{0})=(0,0)$. Since Proposition~\ref{p2} also holds for the asymptotically periodic case, there exist a sequence $(y_{n})_{n}\subset\mathbb{R}^{N}$ and constans $R,\eta>0$ such that
\begin{equation}\label{ej5}
\liminf_{n\rightarrow+\infty}\int_{B_{R}(y_{n})}(u_{n}^{2}+v_{n}^{2})\,\ud x\geq\eta>0.
\end{equation}
Let us consider the shift sequence $(\tilde{u}_{n}(x),\tilde{v}_{n}(x))=(u_{n}(x+y_{n}),v_{n}(x+y_{n}))$. Note that $(\tilde{u}_{n},\tilde{v}_{n})_{n}$ is not necessarily a $(Ce)_{c}$ sequence for $I_{\lambda}$. On the other hand, using \ref{v3} we can check that $(\tilde{u}_{n},\tilde{v}_{n})_{n}$ is bounded in $E_{p}$. Hence, up to a subsequence, we have that
 \begin{itemize}
 	\item $(\tilde{u}_{n},\tilde{v}_{n})\rightharpoonup(\tilde{u}_{0},\tilde{v}_{0})$ weakly in $E_{p}$;
 	\item $\tilde{u}_{n}\rightarrow \tilde{u}_{0}$ strongly in $L^{r}_{loc}(\mathbb{R}^{N})$, for all $2\leq r<2^{*}_{s_{1}}$;
 	\item $\tilde{v}_{n}\rightarrow \tilde{v}_{0}$ strongly in $L^{s}_{loc}(\mathbb{R}^{N})$, for all $2\leq s<2^{*}_{s_{2}}$;
 	\item $\tilde{u}_{n}(x)\rightarrow \tilde{u}_{0}(x)$ and $\tilde{v}_{n}(x)\rightarrow \tilde{v}_{0}(x)$, almost everywhere in $\mathbb{R}^{N}$.
 \end{itemize}
Thus, we can deduce from \eqref{ej5} that $(\tilde{u}_{0},\tilde{v}_{0})\neq(0,0)$.

\vspace{0,3cm}

\noindent \textit{Claim.} $I_{\lambda,p}'(\tilde{u}_{0},\tilde{v}_{0})=0$.

\vspace{0,3cm}

By density, it suffices to prove that
 \begin{equation}\label{ej8}
  \langle I_{\lambda,p}'(\tilde{u}_{0},\tilde{v}_{0}),(\varphi,\psi)\rangle=0, \quad \mbox{for all} \hspace{0,2cm} (\varphi,\psi)\in C^{\infty}_{0}(\mathbb{R}^{N})\times C^{\infty}_{0}(\mathbb{R}^{N}).
 \end{equation}
For $(\varphi,\psi)\in C^{\infty}_{0}(\mathbb{R}^{N})\times C^{\infty}_{0}(\mathbb{R}^{N})$ we denote $(\varphi_{n}(x),\psi_{n}(x))=(\varphi(x-y_{n}),\psi(x-y_{n}))$. Thus, using Lebesgue Dominated Convergence Theorem, we can deduce that
 \begin{equation}\label{ej6}
  \langle I_{\lambda,p}'(\tilde{u}_{0},\tilde{v}_{0}),(\varphi,\psi)\rangle=\langle I_{\lambda,p}'(u_{n},v_{n}),(\varphi_{n},\psi_{n})\rangle +o_{n}(1).
 \end{equation}
Moreover, we have that
 \begin{align*}
  \langle I_{\lambda,p}'(u_{n},v_{n}),(\varphi_{n},\psi_{n})\rangle=\langle I_{\lambda}'(u_{n},v_{n}),(\varphi_{n},\psi_{n})\rangle+\int_{\mathbb{R}^{N}}(V_{1,p}(x)-V_{1}(x))u_{n}\varphi_{n}\,\ud x+\\
                                                          +\int_{\mathbb{R}^{N}}(V_{2,p}(x)-V_{2}(x))v_{n}\psi_{n}\,\ud x-\int_{\mathbb{R}^{N}}(\lambda_{p}(x)-\lambda(x))(u_{n}\psi_{n}+v_{n}\varphi_{n})\,\ud x.
 \end{align*}
By using Lemma~\ref{conv} we conclude that
 \begin{equation}\label{ej7}
  \langle I_{\lambda,p}'(u_{n},v_{n}),(\varphi_{n},\psi_{n})\rangle=\langle I_{\lambda}'(u_{n},v_{n}),(\varphi_{n},\psi_{n})\rangle+o_{n}(1).
 \end{equation}
Thus, since $\|(\varphi_{n},\psi_{n})\|_{E_{p}}=\|(\varphi,\psi)\|_{E_{p}}$, it follows from \ref{v3} and \eqref{css} that
 \begin{equation}\label{ej9}
  \langle I_{\lambda}'(u_{n},v_{n}),(\varphi_{n},\psi_{n})\rangle\leq \|I_{\lambda}'(u_{n},v_{n})\|_{E^{*}}\|(\varphi_{n},\psi_{n})\|_{E}\leq\|I_{\lambda}'(u_{n},v_{n})\|_{E^{*}}\|(\varphi_{n},\psi_{n})\|_{E_{p}}\rightarrow0,
 \end{equation}
as $n\rightarrow\infty$. Therefore, combining \eqref{ej6}, \eqref{ej7} and \eqref{ej9} we get \eqref{ej8} and the claim is proved. \qed

Hence, using Lemma~\ref{gs}, \eqref{css}, Fatou's Lemma and the preceding assertion, we have
\[
\begin{alignedat}{2}	
        c + o_{n}(1) & =  I_{\lambda}(u_{n},v_{n})-\frac{1}{2}\langle I_{\lambda}'(u_{n},v_{n}),(u_{n},v_{n})\rangle\\
                     & =  \frac{1}{2}\int_{\mathbb{R}^{N}}(f_{1}(u_{n})u_{n}-2F_{1}(u_{n}))\,\ud x+\frac{1}{2}\int_{\mathbb{R}^{N}}(f_{2}(v_{n})v_{n}-2F_{2}(v_{n}))\,\ud x\\
                     & =  \frac{1}{2}\int_{\mathbb{R}^{N}}(f_{1}(\tilde{u}_{n})\tilde{u}_{n}-2F_{1}(\tilde{u}_{n}))\,\ud x+\frac{1}{2}\int_{\mathbb{R}^{N}}(f_{2}(\tilde{v}_{n})\tilde{v}_{n}-2F_{2}(\tilde{v}_{n}))\,\ud x\\
                     & \geq  \frac{1}{2}\int_{\mathbb{R}^{N}}(f_{1}(\tilde{u}_{0})\tilde{u}_{0}-2F_{1}(\tilde{u}_{0}))\,\ud x+\frac{1}{2}\int_{\mathbb{R}^{N}}(f_{2}(\tilde{v}_{0})\tilde{v}_{0}-2F_{2}(\tilde{v}_{0}))\,\ud x+o_{n}(1)\\
                     & =  I_{\lambda,p}(\tilde{u}_{0},\tilde{v}_{0})+o_{n}(1),
\end{alignedat}
\]
which implies that $I_{\lambda,p}(\tilde{u}_{0},\tilde{v}_{0})\leq c$. Thus, using \ref{v3} and \ref{v4} we conclude that
 $$
  c\leq \max_{t\geq0}I_{\lambda}(t\tilde{u}_{0},t\tilde{v}_{0})\leq \max_{t\geq0}I_{\lambda,p}(t\tilde{u}_{0},t\tilde{v}_{0})=I_{\lambda,p}(\tilde{u}_{0},\tilde{v}_{0})\leq c.
 $$
Therefore, there exists $\gamma_{0}\in\Gamma$ such that $c=\max_{t\in[0,1]}I_{\lambda}(\gamma_{0}(t))$. It follows from Theorem~\ref{elves} that $I$ possesses a nontrivial critical point $(u_{0},v_{0})\in K_{c}\cap\gamma_{0}([0,1])$. By using Lemma~\ref{nehari}~$(i)$, we note that
 $
  I_{\lambda}(u_{0},v_{0})=c\leq c_{\mathcal{N}_{\lambda}}.
 $
Since $(u_{0},v_{0})\in\mathcal{N}_{\lambda}$, it follows that $(u_{0},v_{0})$ is a ground state solution for System~\eqref{paper1jap}. Arguing as in the proof of Theorem~\ref{A}, we conclude that if $\lambda$ is positive, then $u_{0}$ and $v_{0}$ are positive.
\end{proof}



\section{Proof of Theorem~\ref{C}}\label{proofC}

This Section is devoted to the proof of Theorem~\ref{C}. For this purpose we obtain the following Lemma which study the sign of the ground state solution of System~\eqref{paper1jap} in the limit case, that is, when $\lambda=0$.

\begin{lemma}\label{sinal}
	Let $(u_{0},v_{0})\in E$ be a ground state solution for System~\eqref{paper1jap} with $\lambda=0$. Then either $u_{0}>0$, $v_{0}\equiv0$ or $u_{0}\equiv0$ and $v_{0}>0$.
\end{lemma}
\begin{proof}
	 Let $I_{0}$ be the energy functional associated to System~\eqref{paper1jap} with $\lambda=0$. Notice that
	  \[
	   \left\{
	    \begin{array}{rl}
	     I_{0}(u,v)>I_{0}(u,0), & \mbox{for all } (u,v)\in E \mbox{ with } v\neq0,\\
	     I_{0}(u,v)>I_{0}(0,v), & \mbox{for all } (u,v)\in E \mbox{ with } u\neq0.
	    \end{array}
	   \right.
	  \]
	 Since $(u_{0},v_{0})$ is a ground state, that is, $I_{0}(u_{0},v_{0})$ has minimum energy among all nontrivial solutions, we conclude that either $u_{0}=0$ or $v_{0}=0$. By using similar ideas of the preceding sections jointly with the fact that $f_{i}(t)=0$ for $s\leq0$ and $i=1,2$, we conclude that either $u_{0}>0$ or $v_{0}>0$.
\end{proof}

From now on, for any $n\in\mathbb{N}$,  we consider $(u_{\lambda_{n}},v_{\lambda_n})_{n}\in\mathcal{N}_{\lambda_{n}}$ the positive ground state solution for System~\eqref{paper1jap} with $\lambda=\lambda_{n}$. Suppose that $\|\lambda_{n}\|_{\infty}\rightarrow0$ as $n\rightarrow\infty$. It follows from Proposition~\ref{conv} that $c_{\lambda_{n}}\rightarrow c_{0}$, as $n\rightarrow+\infty$ where $c_{0}$ is the least energy level for the System \eqref{paper1jap} with $\lambda = 0$. Thus, the sequence $(c_{\lambda_{n}})_{n}$ is bounded. By similar ideas to used in Proposition~\ref{p1}, it is not difficulty to prove that $(u_{\lambda_{n}},v_{\lambda_n})_{n}$ is bounded in $E$. Thus, up to a subsequence, $(u_{\lambda_{n}},v_{\lambda_n})\rightharpoonup (U_{0},V_{0})$ weakly in $E$. We claim that $(U_{0},V_{0})\neq(0,0)$. Suppose by contradiction that $(U_{0},V_{0})=(0,0)$. Notice that $(u_{\lambda_{n}},v_{\lambda_{n}})_{n}$ can not converge stronlgy to $(0,0)$. Indeed, in this case, there holds
	\[
	 0<c_{0}=\lim_{n\rightarrow+\infty}c_{\lambda_{n}}=\lim_{n\rightarrow+\infty}I_{\lambda_{n}}(u_{\lambda_{n}},v_{\lambda_{n}})=0,
	\]
which is not possible. Thus, in light of Proposition~\ref{p2}, there exist a sequence $(y_{n})_{n}\subset\mathbb{R}^{N}$ and constants $R,\eta>0$ such that $|y_{n}|\rightarrow\infty$ as $n\rightarrow\infty$, and
	\begin{equation}\label{ejj1}
	 \liminf_{n\rightarrow+\infty}\int_{B_{R}(y_{n})}(u_{\lambda_{n}}^{2}+v_{\lambda_{n}}^{2})\,\ud x\geq\eta>0.
	\end{equation}
We denote $(\tilde{u}_{\lambda_{n}}(x),\tilde{v}_{\lambda_{n}}(x))=(u_{\lambda_{n}}(x+y_{n}),v_{\lambda_{n}}(x+y_{n}))$. By using \ref{v3} it follows that $(\tilde{u}_{\lambda_{n}},\tilde{v}_{\lambda_{n}})_{n}$ is bounded in $E_{p}$. Hence, up to a subsequence, $(\tilde{u}_{\lambda_{n}},\tilde{v}_{\lambda_n})\rightharpoonup (\tilde{U}_{0},\tilde{V}_{0})$ weakly in $E_{p}$. Arguing as in \eqref{ej8}, it is not hard to conclude that $I_{\lambda_{n},p}^{\prime}(\tilde{U}_{0},\tilde{V}_{0})=0$. Therefore, we can deduce that
 \begin{alignat*}{2}
  c_{\mathcal{N}_{\lambda_{n},p}} & \leq I_{\lambda_{n},p}(\tilde{U}_{0},\tilde{V}_{0})-\frac{1}{2}\langle I_{\lambda_{n},p}^{\prime}(\tilde{U}_{0},\tilde{V}_{0}),(\tilde{U}_{0},\tilde{V}_{0})\rangle\\
                                  & = \frac{1}{2}\int_{\mathbb{R}^{N}}(f_{1}(\tilde{U}_{0})\tilde{U}_{0}-2F_{1}(\tilde{U}_{0}))\,\mathrm{d}x+\frac{1}{2}\int_{\mathbb{R}^{N}}(f_{2}(\tilde{V}_{0})\tilde{V}_{0}-2F_{2}(\tilde{V}_{0}))\,\mathrm{d}x\\
                                  & \leq \frac{1}{2}\int_{\mathbb{R}^{N}}(f_{1}(\tilde{u}_{\lambda_{n}})\tilde{u}_{\lambda_{n}}-2F_{1}(\tilde{u}_{\lambda_{n}}))\,\mathrm{d}x+\frac{1}{2}\int_{\mathbb{R}^{N}}(f_{2}(\tilde{v}_{\lambda_{n}})\tilde{v}_{\lambda_{n}}-2F_{2}(\tilde{v}_{\lambda_{n}}))\,\mathrm{d}x+o_{n}(1)\\
                                  & = \frac{1}{2}\int_{\mathbb{R}^{N}}(f_{1}(u_{\lambda_{n}})u_{\lambda_{n}}-2F_{1}(u_{\lambda_{n}}))\,\mathrm{d}x+\frac{1}{2}\int_{\mathbb{R}^{N}}(f_{2}(v_{\lambda_{n}})v_{\lambda_{n}}-2F_{2}(v_{\lambda_{n}}))\,\mathrm{d}x+o_{n}(1)\\
                                  & = I_{\lambda_{n}}(u_{\lambda_{n}},v_{\lambda_{n}})-\frac{1}{2}\langle I_{\lambda_{n}}^{\prime}(u_{\lambda_{n}},v_{\lambda_{n}}),(u_{\lambda_{n}},v_{\lambda_{n}})\rangle+o_{n}(1)\\
                                  & = c_{\mathcal{N}_{\lambda_{n}}}+o_{n}(1),
 \end{alignat*}
which contradicts Lemma~\ref{nehari}~(ii). Therefore, $(U_{0},V_{0})\neq(0,0)$. Now, notice that
 \[
  \langle I_{0}^{\prime}(u_{\lambda_{n}},v_{\lambda_{n}}),(\varphi,\psi)\rangle=\langle I_{\lambda_{n}}^{\prime}(u_{\lambda_{n}},v_{\lambda_{n}}),(\varphi,\psi)\rangle+\int_{\mathbb{R}^{N}}\lambda_{n}(x)(u_{\lambda_{n}}\psi+v_{\lambda_{n}}\varphi)\,\mathrm{d}x=o_{n}(1),
 \]
for all $(\varphi,\psi)\in C^{\infty}_{0}(\mathbb{R}^{N})\times C^{\infty}_{0}(\mathbb{R}^{N})$. Moreover, one has
 \begin{align*}
  \langle I_{0}^{\prime}(U_{0},V_{0}),(\varphi,\psi)\rangle = \langle I_{0}^{\prime}(u_{\lambda_{n}},v_{\lambda_{n}}),(\varphi,\psi)\rangle-(u_{\lambda_{n}}-U_{0},\varphi)_{E_{1}}-(v_{\lambda_{n}}-V_{0},\psi)_{E_{2}}\\-\int_{\mathbb{R}^{N}}(f_{1}(u_{n})-f_{1}(U_{0}))\varphi\,\mathrm{d}x-\int_{\mathbb{R}^{N}}(f_{2}(v_{n})-f_{2}(V_{0}))\psi\,\mathrm{d}x=o_{n}(1),
 \end{align*}
for all $(\varphi,\psi)\in C^{\infty}_{0}(\mathbb{R}^{N})\times C^{\infty}_{0}(\mathbb{R}^{N})$, which implies that $I_{0}^{\prime}(U_{0},V_{0})=0$. Hence, it follows that $I_{0}(U_{0},V_{0})\geq c_{0}$. On the other hand, we have
 \begin{alignat*}{2}
  	I_{0}(U_{0},V_{0}) & = I_{0}(U_{0},V_{0})-\frac{1}{2}\langle I_{0}^{\prime}(U_{0},V_{0}),(U_{0},V_{0})\rangle\\
  	& = \frac{1}{2}\int_{\mathbb{R}^{N}}(f_{1}(U_{0})U_{0}-2F_{1}(U_{0}))\,\mathrm{d}x+\frac{1}{2}\int_{\mathbb{R}^{N}}(f_{2}(V_{0})V_{0}-2F_{2}(V_{0}))\,\mathrm{d}x\\
  	& \leq \frac{1}{2}\int_{\mathbb{R}^{N}}(f_{1}(u_{\lambda_{n}})u_{\lambda_{n}}-2F_{1}(u_{\lambda_{n}}))\,\mathrm{d}x+\frac{1}{2}\int_{\mathbb{R}^{N}}(f_{2}(v_{\lambda_{n}})v_{\lambda_{n}}-2F_{2}(v_{\lambda_{n}}))\,\mathrm{d}x+o_{n}(1)\\
  	& = I_{\lambda_{n}}(u_{\lambda_{n}},v_{\lambda_{n}})-\frac{1}{2}\langle I_{\lambda_{n}}^{\prime}(u_{\lambda_{n}},v_{\lambda_{n}}),(u_{\lambda_{n}},v_{\lambda_{n}})\rangle+o_{n}(1)\\
  	& = c_{\mathcal{N}_{\lambda_{n}}}+o_{n}(1).
 \end{alignat*}
Therefore $I_{0}(U_{0},V_{0})=c_{0}$ and $(U_{0},V_{0})$ is a ground state solution for System~\eqref{paper1jap} with $\lambda=0$. By using Lemma~\ref{sinal}, we conclude that one of the following conclusions holds:
\begin{itemize}
	\item[(i)] $V_{0}\equiv0$ and $U_{0}$ is a positive ground state of
	\begin{equation}\label{ed1}
	(-\Delta)^{s_{1}}u+V_{1}(x)u=f_{1}(u), \quad x\in\mathbb{R}^{N}.
	\end{equation}
\end{itemize}
\begin{itemize}
	\item[(ii)] $U_{0}\equiv0$ and $V_{0}$ is a positive ground state of
	\begin{equation}\label{ed2}
	(-\Delta)^{s_{2}}v+V_{2}(x)v=f_{2}(v), \quad x\in\mathbb{R}^{N}.
	\end{equation}
\end{itemize}
In particular, $c_{0} = \min \{c_{\mathcal{N}_{V_{1}}}, c_{\mathcal{N}_{V_{2}}}\}$ where $c_{\mathcal{N}_{V_{1}}}$ and $c_{\mathcal{N}_{V_{2}}}$ denotes the least energy level for the scalar equations \eqref{ed1} and \eqref{ed2}, respectively. The sets $\mathcal{N}_{V_{1}}$ and $\mathcal{N}_{V_{2}}$ denotes also the Nehari manifold for the scalar equations \eqref{ed1} and \eqref{ed2}, respectively. This finishes the proof of Theorem~\ref{C}.






\end{document}